\documentclass[amssymb,amscd,11pt,20pt]{amsart}
\usepackage{amsmath,amsthm}
\usepackage{amsfonts}
\usepackage{amscd}
\usepackage{amscd}
\usepackage[all]{xy}

 \newtheorem{thm}{Theorem}[subsection]
 \newtheorem{cor}[thm]{Corollary}
 \newtheorem{lem}[thm]{Lemma}
 \newtheorem{prop}[thm]{Proposition}
 \theoremstyle{definition}
 \newtheorem{defn}[thm]{Definition}
 \theoremstyle{remark}
 \newtheorem{rem}[thm]{Remark}
 \newtheorem{ejem}[thm]{Example}
 \newtheorem{ejems}[thm]{Examples}

\numberwithin{equation}{subsection}

\newcommand{\U}{{\mathcal U}}

\newcommand{\ZZ}{{\mathbb Z}}
\newcommand{\QQ}{{\mathbb Q}}
\newcommand{\Lc}{{\mathcal L}}

\newcommand{\T}{{\mathcal T}}
\newcommand{\LL}{{\mathbb L}}
\newcommand{\RR}{{\mathbb R}}
\newcommand{\Sc}{{\mathbb S}}

\DeclareMathOperator{\OOgamma}{{L}}

\newcommand{\wh}{\widehat}
\DeclareMathOperator{\limi}{{lim}}
\newcommand{\plim}[1]{\,\underset{#1}{\underset{\leftarrow}{\limi}}\,}
\newcommand{\ilim}[1]{\,\underset{#1}{\underset{\rightarrow}{\limi}}\,}
\newcommand{\punto}{{\scriptscriptstyle \bullet}}

\newcommand{\enumera}{\begin{enumerate}}
\newcommand{\eenumera}{\end{enumerate}}

 \DeclareMathOperator{\Ima}{{Im}}
\DeclareMathOperator{\Hom}{{Hom}}
\DeclareMathOperator{\DHom}{\underline{H\widehat{om}}}
\DeclareMathOperator{\HHom}{\underline{Hom}}

\DeclareMathOperator{\id}{{id}}

\begin{document}

\title[Homology  and Cohomology of Finite Spaces]
 {Homology  and Cohomology of Finite Spaces}

\author{ V. Carmona S\'anchez}\author{ C. Maestro P\'erez}\author{  F. Sancho de Salas}\author{ J.F. Torres Sancho}

\address{ \newline V\'ictor Carmona S\'anchez\newline Departamento de \'algebra\newline
Universidad de Sevilla-IMUS\newline   Avda Reina Mercedes s/n, Edificio Celestino Mutis,
41012 Sevilla\newline  Spain}
\email{vcarmona1@us.es}

\address{ \newline Carlos Maestro P\'erez\newline Institut fur Mathematik\newline
Humboldt-Universitat zu Berlin - Berlin Mathematical School (BMS)\newline Rudower Chaussee 25 (Johann von Neumann Haus) \\
12489 Berlin\newline  Germany}
\email{maestroc@mathematik.hu-berlin.de}

\address{ \newline Fernando Sancho de Salas\newline Departamento de
Matem\'aticas\newline
Universidad de Salamanca\newline  Plaza de la Merced 1-4\\
37008 Salamanca\newline  Spain}
\email{fsancho@usal.es}

\address{ \newline Juan Francisco Torres Sancho\newline Departamento de
Matem\'aticas\newline
Universidad de Salamanca\newline  Plaza de la Merced 1-4\\
37008 Salamanca\newline  Spain}
\email{juanfran24@usal.es}



\thanks {The third author was supported by research project MTM2017-86042-P (MEC)}




\begin{abstract} We provide base change theorems, projection formulae and Verdier duality for both cohomology and homology in the context of finite topological spaces.
\end{abstract}

\maketitle

\section*{Introduction}

Finite topological spaces or finite posets are an already classic object of study, ever since the
work of Alexandroff, Stong, McCord, and so on; as well as a current area of interest (see  \cite{Barmak} and references
therein). For a while, let us consider locally compact and separated spaces.


Let $f\colon X\to Y$ be a proper map between locally compact and separated topological spaces. The cohomology of $f$ (i.e., the higher direct images) satisfies (see \cite{Iversen}):

\medskip\begin{enumerate}

\item Base change theorem: For any sheaf $F$ of abelian groups on $X$ and any $y\in Y$, one has an isomorphism
\[ (R^if_*F)_y\overset\sim\to H^i(f^{-1}(y),F_{\vert f^{-1}(y)}).\]

\item Projection formula: for any $F\in D^-(X)$ and any $G\in D^-(Y)$ one has an isomorphism
\[ \RR f_* F\overset\LL\otimes G\overset\sim\to \RR f_*(F\overset\LL\otimes f^{-1}G).\]

\item Verdier duality: There exists a functor $f^!\colon D^+(Y)\to D^+(X)$ satisfying:
\[ \RR\Hom^\punto(\RR f_* F,G) = \RR\Hom^\punto(F,f^!G) \] for any $F\in D^+(X)$, $G\in D^+(Y)$.

\item[(3.1)] Local isomorphism of duality: the isomorphism (3) can be sheafified to an isomorphism 
\[ \RR\HHom^\punto(\RR f_* F,G) = \RR f_*\RR\HHom^\punto(F,f^!G). \]

\end{enumerate}

A first aim of this paper is to study the validity of these results in the context of finite topological spaces. Thus, let $f\colon X\to Y$ be a continuous map between finite topological spaces (or finite posets). Then Verdier duality (3) still holds (\cite{Navarro}). However (1), (2) and (3.1) are no longer true, even if we assume that $f$ is proper (i.e., universally closed). Properness is too weak in this context (it is equivalent to being closed). It is a necessary but not sufficient condition for having (1), (2) and (3.1).  A stronger notion is necessary. This notion is that of a {\it cohomologically proper} map: We say that $f\colon X\to Y$ is cohomologically proper (c-proper for short) if it is a closed map and for any $x\in X$ the map $f_{\vert{C_x}}\colon C_x\to C_y$ has homologically trivial fibers, where $C_x, C_y$ denotes the closure of $x$ and $y$ respectively. We prove that a c-proper map satisfies (1), (2) and (3.1) and conversely, any of them characterizes a c-proper map. These results (and others characterizing c-proper maps) are given in Theorems \ref{c-proper&basechange}, \ref{hsatprojform} and  \ref{projformulacharacterization}.

A second aim of the paper is to develop these results for homology. One of the most striking properties of a  finite topological space $X$ is the equivalence between the category of sheaves of abelian groups on $X$  and the category of cosheaves (over closed subsets) of abelian groups on $X$, thanks to the fact that both of these categories are in turn equivalent to the category of functors from $X$ to abelian groups (which will be referred to in this paper as abelian data on $X$). This has the following fundamental implications. On an arbitrary topological space, cohomology is defined for any sheaf and studied within the framework of the theory of derived functors, whereas homology is defined only for constant or locally constant coefficients and thus does not fit within such a framework. However, on finite topological spaces one may in fact define the homology and cohomology of any abelian data (or sheaf, or cosheaf) and both of these constructions are developed within the framework of the theory of derived functors. In other words, on a finite topological space it is possible to consider the homology groups $H_i(X,F)$ with coefficients on any sheaf $F$, and in this case  $H_i(X,\quad)$ is the $i$-th (left-)derived functor of $H_0(X,\quad)$. This was already pointed out by  Deheuvels (\cite{Deheuvels}). Moreover, just as the functor of sections of a sheaf admits a relative version  (the direct image $f_*$), so too does the functor of cosections of a cosheaf admit a relative version $f_!$ (which is the direct image of cosheaves), and again, just as cohomology admits a relative version  (the higher direct images $R^if_*$), so too does homology admit a relative version  $L_if_!$. Thus, one can formulate base change or projection formula morphisms for homology and to study when these morphisms are isomorphisms. This leads to the notion of an {\it homologically open} map (h-open for short), which is the analogous for homology to that of a c-proper map, and we obtain the same results for homology in Theorems \ref{h-open&basechange} and \ref{hsatprojform} (we leave aside Verdier duality for homology for the moment, and will treat it below). Base change theorem for cohomology (resp. for homology) is related to the theory of cohomological (resp. homological) fiber theorems    (see \cite{BjörnerWachsWelker} and \cite{Quillen}); as an example, see Corollary \ref{3.2.10}.

A great part of the paper is dedicated to duality (section \ref{4}). There are several types of duality to be considered.  Subsection 4.1  is devoted to the duality between homology and cohomology. We generalize to a relative situation and, for any complex of abelian data, to the standard duality between cohomology and homology with coefficientes on a locally constant sheaf. Subsection 4.2 is devoted to  topological duality, that is, to relating  the cohomology of  complexes of abelian data on $X$ to the cohomology of complexes of abelian data on the dual space $\widehat X$. The main results are  Theorem \ref{topduality}, Corollary \ref{topduality2} and Theorem \ref{topdualityrel}, that generalize the well known fact that a space and its dual have the same cohomology and homology groups with coefficients on a constant sheaf.

In subsection 4.3 we study Grothendieck-Verdier duality for first cohomology and then homology for a morphism $f\colon X\to Y$. Some related duality results can also be found in \cite{Curry}. Verdier duality for cohomology deals with the existence of a right adjoint $f^!$ of the functor $\RR f_*$. This is the global form of duality, which is already done in \cite{Navarro}. We shall give an explicit description of the dualizing complex $D_X$ in Theorem \ref{ExplicitDualizing} (a related result may be found in \cite{Curry}). As we mentioned above, we also study the validity of the local form of duality in Theorem \ref{projformulacharacterization}.


Now, let us comment on the results obtained about Grothendieck-Verdier duality for homology, which we have named as co-duality. The first result is the existence of a left adjoint $f^\#$ of the functor $\LL f_!$ (Theorem \ref{coduality}). This is the global form of co-duality. In particular, we obtain a relative co-dualizing complex $D^{X/Y}:=f^\#\ZZ$ and a co-dualizing complex $D^X$ (when $Y$ is a point). We shall give an explicit description of the complex $D^X$ in Theorem \ref{ExplicitCodualizing}. It was a surprising fact for us to find that there is not a natural local form of co-duality (in contrast with duality). What  still holds is the local structure on $Y$ of the functor $f^\#$, in the sense that there is a morphism (Proposition \ref{localcodualizing}), though it is not an isomorphism in general. The homological version of Theorem \ref{projformulacharacterization} is given in Theorem \ref{coprojformulacharacterization}. This theorem  says in particular that the functor $f^\#$ is local on $Y$ if and only if $f$ is  h-open.  Notice that there is not a perfect symmetry between the cohomological statements of Theorem \ref{projformulacharacterization} and the homological ones in Theorem \ref{coprojformulacharacterization} (in contrast with  Theorems \ref{c-proper&basechange}-\ref{h-open&basechange} and Theorem \ref{hsatprojform}). For example: while the local form of the adjunction between $\RR f_*$ and $f^{-1}$   holds for any $f$, the local form of the adjunction between $\LL f_!$ and $f^{-1}$ holds if and only if $f$ is h-open (Theorem \ref{coprojformulacharacterization}).

Finally, in Theorem \ref{dualizing-codualizing} and Proposition \ref{perfect} we see the relation between the dualizing  and co-dualizing complexes of a space $X$.

Sections 1. and 2. have no original results and are well known for specialists on the subject. We have included them in order to fix notations and make the exposition more self-contained. In section 1., we recall some basic notions related to sheaves and cosheaves on a finite topological space and their equivalence to abelian data (Theorem \ref{sheaf=abeliandata=cosheaf}), as well as some facts about direct and inverse images, supporting on locally closed subspaces, homology and cohomology of abelian
data and standard resolutions. In section 2., we collect some standard results about the derived category of abelian data on a finite topological space, locally constant complexes, perfect complexes and several adjunction formulas.

Many of the notions and results that are present in this paper can be extended to non-finite posets. However, since the finiteness hypothesis is required for the main results obtained here, we have preferred to focus our discussion solely on finite topological spaces.

\section{Basics} 

We shall give here some basic results about sheaves, cosheaves, direct and inverse images, homology, cohomology and standard resolutions (see also \cite[Sections 2., 3. and 6.]{Curry}).
\medskip

Let $X$  be a finite topological space (for the sake of brevity, we shall just say a finite space to refer to a finite topological space). For each $p\in X$ let us denote
\[\aligned  U_p&=\text{ smallest open subset containing } p,\\ C_p &=\text{ smallest closed subset containing } p =\text{ closure of } p.\endaligned\]

\noindent (Alexandroff): We have a preorder on $X$: $p\leq q$ iff $C_p\subseteq C_q$ (equivalently  $U_p\supseteq U_q$). Thus
\[ U_p=\{ q\in X: q\geq p\},\qquad C_p=\{ q\in X: q\leq p\}.\]

If $X,Y$ are two finite  spaces, then a map $f\colon X\to Y$ is continuous if and only if it is monotone: $p\leq q$ implies that $f(p)\leq f(q)$.

\begin{defn} The {\it dimension} of a finite space $X$ is the maximum of the lengths of the chains $x_0< x_1 <\cdots < x_n$. It is also called {\it length} of $X$ in the literature.
\end{defn}

\begin{defn} Let $X$ be a finite topological space. The {\it dual space} $\wh X$ is the same set as $X$, but with the opposite topology: an open subset of $\wh X$ is a closed subset of $X$.  For any $x\in X$ we shall denote by $\widehat x$ the same element $x$ but considered as an element of $\widehat X$. The preorder in $\wh X$ is the inverse of that of $X$; thus,   $\widehat x\leq \widehat y$ is equivalent to $x\geq y$. Given a map $f\colon X\to Y$, we shall denote $\widehat f\colon\widehat X\to \widehat Y$ the same map but with the opposite topologies on the sets; $f$ is continuous if and only if $\widehat f$ is. 
\end{defn}

\subsection{Abelian data, sheaves and cosheaves}

\begin{defn} An {\it abelian data} $F$ on $X$ is the following data:  an abelian group  $F_p$ for each $p\in X$ and a group homomorphism  $r_{pq}\colon F_p\to F_q$ for each $p\leq q$, satisfying $$r_{pp}=\id, \text{ for any } p\in X,\quad \text{and}\quad r_{ql}\circ r_{pq}=r_{pl} \text{ for any } p\leq q\leq l.$$  If we view $X$ as a category, $F$ is just a functor $$F\colon X\to \{\text{Abelian Groups}\}.$$  Morphisms of abelian data are just morphisms of functors. Thus one has the category of abelian data on $X$, which is abelian. 
\end{defn}

If $\{ F_i\}$ is a direct system of abelian data, the direct limit $\ilim{} F_i$ is the abelian data defined by \[(\ilim{} F_i)_p:=\ilim{} (F_i)_p\] and it coincides with the categorical direct limit. Analogously, the inverse limit $\plim{} F_i$ is defined by $$(\plim{} F_i)_p := \plim{} (F_i)_p$$ and it coincides with the categorical inverse limit. The tensor product $F\otimes F'$ of two abelian data $F,F'$ is defined by $(F\otimes F')_p:=F_p\otimes F'_p$.

\begin{ejem} The finite space with one element will be denoted by $\{ *\}$. An abelian data on $\{ *\}$ is just an abelian group. Thus the category of abelian data on $\{ *\}$ is the category of abelian groups.
\end{ejem}

\begin{defn} Let $S$ be a subspace of $X$, $F$ an abelian data on $X$. We define:
\[ \Gamma(S,F)= \plim{s\in S} F_s,\qquad L(S,F)=  \ilim{s\in S} F_s\] which will be called {\it sections} and {\it cosections} of $F$ on $S$ respectively.
\end{defn}

By definition, one has exact sequences
\[\aligned 0\to \Gamma(S,F)\to \prod_{s\in S} F_s \to\prod_{(s<s')\in S} F_{s'} \\ 
\bigoplus_{(s'<s)\in S} F_{s'}\to \bigoplus_{s\in S} F_s\to  L(S,F)\to  0
\endaligned .\]

Functoriality: $\Gamma(S,F)$ and $L(S,F)$ are functorial on $S$ and $F$. If  $S\subseteq S'$, one has morphisms $\Gamma(S',F)\to \Gamma(S,F)$ and $L(S,F)\to L(S',F)$. If $F\to F'$ is a morphism of abelian data on $X$, then it induces morphisms $\Gamma(S,F)\to\Gamma(S,F')$ and $L(S,F)\to L(S,F')$. Thus, for a fixed subspace $S$, one has a covariant  functor 
\[  \Gamma(S,\quad)\colon \{\text{Abelian data on } X\} \to \{\text{Abelian groups} \}\] which is left exact, and a covariant functor
\[ L(S,\quad)\colon \{\text{Abelian data on } X\} \to \{\text{Abelian groups} \} \]
which is right exact.

\begin{rem} For each $p\in X$ one has: $$\Gamma(U_p,F)=F_p=L(C_p,F).$$
\end{rem} 

%
%

\begin{defn} A {\it sheaf} $F$ on $X$ is a contravariant functor
\[\aligned F\colon \{\text{\rm Open subsets of } X\} &\to \{\text{\rm Abelian groups}\}
\\ U &\rightsquigarrow F(U)\\ V\subseteq U &\rightsquigarrow F(U)\to F(V)\endaligned
\] such that for any open subset $U$ and any open covering   $U=\cup_i  U_i $   the sequence 
\[ 0\to F(U)\to\prod_i F(U_i) \to \prod_{i,j} F(U_i\cap U_j)\] is exact. A morphism of sheaves is just a morphism of functors, i.e., a natural transformation.
\end{defn}

\begin{defn} A {\it cosheaf} $F$ on $X$ is a \underline{covariant} functor
\[\aligned F\colon \{\text{\rm Closed subsets of } X\} &\to \{\text{\rm Abelian groups}\}
\\ C &\rightsquigarrow F(C)\\ C\subseteq C' &\rightsquigarrow F(C)\to F(C')\endaligned
\] such that for any closed subset $C$ and any closed covering   $C=\cup_i  C_i $ the sequence
\[  \bigoplus_{i,j} F(C_i\cap C_j)\to  \bigoplus_i F(C_i) \to  F(C)\to 0\] is exact. A morphism of cosheaves is just a morphism of functors, i.e., a natural transformation.
\end{defn}

\begin{rem} The term cosheaf is usually used to refer to a covariant functor from the category of open sets of $X$ to abelian groups (see \cite[Chapter 5]{Bredon}, or \cite{Curry}). However, in this paper, cosheaf means a functor on closed subsets. Since the category of closed subsets of $X$ coincides with the category of open subsets of $\widehat X$, what we call a cosheaf on $X$ coincides with what is called a cosheaf on $\widehat X$ in the literature.
\end{rem}




\begin{thm}\label{sheaf=abeliandata=cosheaf} Let $X$ be a finite topological space. There are natural equuivalences
\[ \{ \text{\rm Sheaves on } X\} \simeq \{ \text{\rm Abelian data on } X\}   \simeq  \{ \text{\rm Cosheaves on } X\}  \] 
\end{thm}
\begin{proof} Let $F$ be an abelian data. Then $F$ can be thought of as a sheaf by defining $F(U)=\Gamma(U,F)$, for each open subset $U$ of $X$,     and it can be thought of as a cosheaf by defining $F(C)=L(C,F)$, for each closed subset.

Conversely, if $F$ is sheaf, then $F$ can be thought of as an abelian data, by taking $F_p:=F(U_p)$ for each $p\in X$, and, for each $p\leq q$, $r_{pq}:=F(i)$, where $i\colon U_q\to U_p$ is the inclusion.  If $F$ is a cosheaf, then $F$ can be thought of as an abelian data, by taking $F_p:=F(C_p)$ for each $p\in X$, and, for each $p\leq q$, $r_{pq}:=F(j)$, where $j\colon C_p\to C_q$ is the inclusion. 
\end{proof}


\subsection{Direct and inverse images} 

\begin{defn} Let $f\colon X\to Y$ be a continuous map, $F$ an abelian data on $Y$. Then $f^{-1}F$ is the abelian data on $X$
 defined by:
\[(f^{-1}F)_x=F_{f(x)}\]
and the obvious restriction morphisms $r_{pq}$. One says that $f^{-1}F$ is the {\it inverse image} of $F$ by $f$. If we view $F$ as a functor $Y\overset F\to \{\text{Abelian groups}\}$, then $f^{-1}F$ is just the composition $$X\overset f\to Y\overset F\to \{\text{Abelian groups}\}.$$ 
\end{defn}

\begin{rem} If $S$ is a subspace of $X$ and $j\colon S\hookrightarrow X$ is the inclusion, we shall  denote $$F_{\vert S}:= j^{-1}F.$$
\end{rem}

A morphism $F\to F'$ of abelian data on $Y$, induces a morphism $f^{-1} F\to f^{-1} F'$ of abelian data on $X$. If $Y\overset g\to Z$ is another continuous map, then $(g\circ f)^{-1}= f^{-1}\circ g^{-1}$.

\begin{prop} $f^{-1}$ is an exact functor and it commutes with direct and inverse limits.
\end{prop}

\begin{defn} Let $f\colon X\to Y$ be a continuous map, $F$ an abelian data on $X$. One defines:
\begin{enumerate}
\item $f_*F$ as the abelian data on $X$ given by:
\[(f_*F)_y=\Gamma(f^{-1}(U_y),F).\]

\item $f_!F$ as the abelian data on $X$ given by:
\[(f_!F)_y=L(f^{-1}(C_y),F).\]
\end{enumerate}
\end{defn}

\begin{prop} Let $f\colon X\to Y$ be a continuous map, $F$ an abelian data on $X$.
\begin{enumerate} \item For any open subset $V$ of $Y$ one has  $$\Gamma(V, f_*F)=\Gamma(f^{-1}(V),F).$$ 
\item For any closed subset $C$ of $Y$ one has $$L(C, f_!F)=L(f^{-1}(C),F).$$\end{enumerate}
\end{prop}

\begin{ejems} \begin{enumerate}
\item Let $\pi\colon X\to \{*\}$ be the projection onto a point. Then $$\pi_*F=\Gamma(X,F)\quad\text{\rm and}\quad  \pi_!F=L(X,F).$$
\item Let $U$ be an open subset and $j\colon U\hookrightarrow X$ the inclusion. For any abelian data $F$ on $U$, $j_!F$ is the extension by zero out of $U$; i.e.
\[ (j_!F)_{\vert U}=F\quad\text{and}\quad (j_!F)_{\vert X- U}=0.\] In particular, $j_!$ is an exact functor.
\item Let $C$ be a closed subset and $j\colon C\hookrightarrow X$ the inclusion. For any abelian data $F$ on $C$, $j_*F$ is the extension by zero out of $C$; i.e.
\[ (j_*F)_{\vert C}=F\quad\text{and}\quad (j_*F)_{\vert X- C}=0.\] In particular, $j_*$ is an exact functor.
\end{enumerate}
\end{ejems}

A morphism $F\to F'$ of abelian data on $X$, induces  morphisms  $f_* F\to f_* F'$ and  $f_! F\to f_! F'$ of abelian data on $Y$. For any continuous map $Y\overset g\to Z$ one has: $(g\circ f)_*= g_*\circ f_* $ and $(g\circ f)_!= g_!\circ f_! $.

\begin{prop}\label{f_*&f_!} $f_*$ is left exact and commutes with inverse limits. $f_!$ is right exact and commutes with direct limits. 
\end{prop}

\begin{rem} Proposition \ref{f_*&f_!} still holds for arbitrary - not necessarily finite - posets. In the finite case it is also true that $f_*$ commutes with direct limits; however, it is not true that $f_!$ commutes with inverse limits in general, even though $X$ is finite. 
\end{rem}

\begin{thm} $f_*$ is the right adjoint of $f^{-1}$ and $f_!$ is the left adjoint of $f^{-1}$. That is, for any abelian data $F$ on $X$ and $G$ on $Y$:
\[ \aligned \Hom(f^{-1}F, G)&=\Hom (F,f_*G)\\  \Hom(G, f^{-1}F)&=\Hom (f_!G,F).\endaligned\]
\end{thm}

\begin{proof} Let $F$ be an abelian data on $Y$ and let $y\in Y$. If $x\in f^{-1}(C_y)$, then $f(x)\leq y$ and one has a morphism $F_{f(x)}\to F_y$. Taking direct limit on $x\in f^{-1}(C_y)$ one obtains a morphism 
\[ (f_!f^{-1}F)_y\to F_y\] hence a morphism of abelian data $\alpha\colon f_!f^{-1}F \to F$. Then a morphism $G\to f^{-1}F$ induces a morphism $f_!G\to f_!f^{-1}F$, and, by composition with $\alpha$, a morphism $f_!G\to F$. This gives a map $$ \widetilde \alpha\colon \Hom(G, f^{-1}F)\to \Hom (f_!G,F).$$ 

Now let $G$ be an abelian data on $X$. For each $x\in X$ the inclusion $C_x\subseteq f^{-1}(C_{f(x)})$ induces a morphism $L(C_x,G)\to L(f^{-1}(C_{f(x)}),G)$, i.e., a morphism $G_x\to (f^{-1}f_!G)_x$. Thus we have a morphism $\beta\colon G \to  f^{-1}f_!G$. Then, a morphism $f_!G\to F$ induces a morphism $f^{-1}f_!G\to f^{-1} F$ and, by composition with $\beta$, a morphism $G\to f^{-1}F$. This gives a map $$\widetilde\beta \colon\Hom (f_!G,F) \to \Hom(G, f^{-1}F) .$$ One checks that $\widetilde \alpha$ and $\widetilde\beta$ are mutually inverse.
\end{proof}

\begin{prop}\label{1.2.10} Let $f\colon X\to Y$ be a continuous map and let $V\subseteq Y$ (resp. $C\subseteq Y$) be an open (resp. closed) subset. For any abelian data $F$ on $X$, one has
\[ (f_* F)_{\vert V}=(f_V)_* (F_{\vert f^{-1}(V)}),\qquad (f_! F)_{\vert C}=(f_C)_! (F_{\vert f^{-1}(C)}) \] where $f_V:f^{-1}(V)\to V$ is the restriction of $f$ to $f^{-1}(V)$ (resp. $f_C\colon f^{-1}(C)\to C$).
\end{prop}

\begin{proof} Immediate.
\end{proof}

\begin{prop}\label{1.2.11} Let us consider a commutative diagram 
\[\xymatrix{ \overline X\ar[r]^{\bar g} \ar[d]_{\bar f} & X\ar[d]^f\\ \overline Y\ar[r]^g & Y .}\]
For any abelian data $F$ on $X$ one has natural morphisms
\[ g^{-1}f_*F\to {\bar f}_*\, {\bar g}^{-1}F,\qquad  {\bar f}_!\, {\bar g}^{-1}F\to g^{-1}f_!F.\]
\end{prop} 

\begin{proof} The natural morphism $F\to f^{-1}f_!F$ induces ${\bar g}^{-1}F\to {\bar g}^{-1} f^{-1}f_!F= {\bar f}^{-1}g^{-1} f_!F$, hence a morphism ${\bar f}_!\, {\bar g}^{-1}F\to g^{-1}f_!F$.
\end{proof}

\subsection{Constant and locally constant abelian data}

\begin{defn} Let $G$ be an abelian group . The {\it constant abelian data $G$} on $X$ is defined by: $G_p=G$ for any $p\in X$ and $r_{pq}=\id$ for any $p\leq q$. An abelian data $\Lc$ is   {\it locally constant} if for any $p\in X$, $\Lc_{\vert U_p}$ is isomorphic to a constant abelian data  on $U_p$. A locally constant data {\it of finite type} is a locally constant data $\Lc$ such that $\Lc_p$ is a finitely generated abelian group for any $p\in X$.
\end{defn}

Locally constant abelian data are characterized by the following:

\begin{prop} An abelian data $\Lc$ is locally constant if and only if $r_{pq}\colon\Lc_p\to\Lc_q$ is an isomorphism for any $p\leq q$.
\end{prop}

\begin{rem}\begin{enumerate} 
\item If $X$ has a minimum (or a maximum) then any locally constant abelian data on $X$ is constant.

\item If $f\colon X\to Y$ is continuous and  $G$ is the constant abelian data $G$ on  $Y$, then $f^{-1}G$ is the constant abelian data $G$ on $X$, i.e. $f^{-1}G=G$. If $\Lc$ is a locally constant abelian data on $Y$, then $f^{-1}\Lc$ is a locally constant abelian data on $X$.
\end{enumerate}
\end{rem}
 
It is proved in \cite{Sancho} that an abelian data $F$ on $X$ is locally constant if and only if it is quasi-coherent (as a sheaf of $\ZZ$-modules). That is why we shall denote by $\text{\rm Qcoh(X)}$ the category of locally constant abelian data on $X$. A locally constant data is of finite type if and only if it is a coherent $\ZZ$-module. We shall denote  by
$\text{\rm Coh(X)}$  the category of locally constant data of finite type on $X$.

\begin{rem} Let $\# X$ be the number of connected components of $X$. If $G$ is the constant abelian data $G$ on $X$, then
\[\aligned \Gamma(X,G)&= G\times \overset{\# X}\cdots\times G, \\ L(X,G)&= G\oplus \overset{\# X}\cdots\oplus G. \endaligned\]
\end{rem}

%
%


\subsection{Supporting on a locally closed subspace}

\begin{defn} Let $S$ be a locally closed subset of $X$ and $F$ an abelian data on $X$. Then $F_S$ is the abelian data on $X$ given by \[ (F_S)_p=\left\{\aligned F_p\text{ if }p\in S, \\ 0\text{ if }p\notin S.\endaligned\right. \] The restriction morphisms $r_{pq}$ are those of $F$ if $p,q\in S$ and zero otherwise (the associativity $r_{ql}\circ r_{pq}= r_{pl}$ is satisfied because $S$ is locally closed). In other words, if we think $F$ as a sheaf, then $F$ is the sheaf supported on $S$. 
\end{defn}

\begin{ejem} If $j\colon C\hookrightarrow X$ is a closed subset, then $$F_C=j_*j^{-1}F.$$   If $j\colon U\hookrightarrow X$ is an open  subset, then $$F_U=j_!j^{-1}F.$$
\end{ejem}

\begin{prop}\label{propsop1} Let $F$ be an abelian data on $X$ and $G$ an abelian group. 
\begin{enumerate} 
\item For any closed subset $Z$ of $X$
\[ \Hom(F, G_Z)=\Hom(L(Z,F),G).\]
\item For any open subset $U$ of $X$
\[ \Hom(G_U,F)=\Hom(G,\Gamma(U,F))).\]
\end{enumerate}
\end{prop}

\begin{proof} Let  $i\colon Z\hookrightarrow X$ be the inclusion and $\pi\colon Z \to \{*\}$ the projection onto a point.
One has $G_Z=i_*\pi^{-1}G$. By adjunctions, one has
\[ \Hom(F, G_Z)= \Hom(F, i_*\pi^{-1}G)= \Hom( F_{\vert Z}, \pi^{-1} G )=   \Hom(\pi_! F_{\vert Z},G).\] One concludes because $ L(Z, F_{\vert Z})=L(Z, F)$.
\end{proof}


\begin{defn} Let  $F$ be an abelian data on $X$. \begin{enumerate} \item For any closed subset $Y\subseteq X$, we shall denote by $\Gamma_Y(X,F)$ the kernel of the morphism $\Gamma(X,F)\to \Gamma (X-Y,F)$. 
\item For any open subset $U\subseteq X$, we shall denote by $L^U(X,F)$ the cokernel of the morphism $L(X-U,F)\to L (X,F)$. 
\end{enumerate}
\end{defn}

\begin{prop}\label{propsop2}  Let $F$ be an abelian data on $X$ and $G$ an abelian group. 
\begin{enumerate} 
\item For any closed subset $Z$ of $X$
\[ \Hom(  G_Z,F)=\Hom(G, \Gamma_Z(X,F) ).\]
\item For any open subset $U$ of $X$
\[ \Hom(F,G_U )=\Hom(L^U(X,F),G).\]
\end{enumerate}
\end{prop}

\begin{proof} For any closed subset $Z$ one has an exact sequence 
\[ 0\to G_{X-Z}\to G\to G_Z\to 0.\] Taking $\Hom(\quad,F)$ and applying Proposition \ref{propsop1}, one concludes (1). Analogously, taking $\Hom(F,\quad )$ and applying Proposition \ref{propsop1}, one concludes (2).
\end{proof}

\subsection{Homology and cohomology}

In order to derive the sections and cosections functors, we need the following basic result.

\begin{thm} The category of abelian data on $X$ has enough injectives and projectives.
\end{thm}


If $I$ is an injective abelian data on $X$, then $I_{\vert U}$ is an injective abelian data on $U$, for any open subset $U$, and $f_*I$ is an inyective abelian data on $Y$ for any continuous map $f\colon X\to Y$. Analogously, f $P$ is a projective abelian data on $X$, then $P_{\vert C}$ is a projective abelian data on $C$, for any closed subset $C$, and $f_!P$ is a proyective abelian data on $Y$ for any continuous map $f\colon X\to Y$.

\begin{defn} We shall denote by $H^i(X,\quad)$ the right derived functors of $\Gamma(X,\quad)$ and by $H_i(X,\quad)$ the left derived functors of $L(X,\quad)$. More generally, $R^if_*$ and $L_if_!$ are the left and right derived functors of $f_*$and $f_!$ respectively. One has
\[ (R^if_*F)_y=H^i(f^{-1}(U_y),F),\qquad (L_if_!F)_y=H_i(f^{-1}(C_y),F).\]
\end{defn}

\begin{defn} Let $F$ be an abelian data on $X$. $F$ is called {\it flasque} if $\Gamma(X,F)\to \Gamma(U,F)$ is surjective for any open subset $U$. $F$ is called {\it coflasque} if $L(C,F)\to L(X,F)$ is injective for any closed subset $C$.
\end{defn}

\begin{prop} If $F$ is flasque, then $F_{\vert U}$ is flasque for any open subset $U$. If $F$ is coflasque, then $F_{\vert Z}$ is coflasque for any closed subset $Z$. If $f\colon X\to Y$ is a continuous map and $F$ is an abelian data on $X$, then $F$ flasque implies that $f_*F$ is flasque and $F$ coflasque implies that  $f_!F$ is coflasque. \end{prop}

\begin{prop}If $F$ is flasque, then  $H^i(X,F)=0$ for $i>0$. If $F$ is coflasque, then  $H_i(X,F)=0$ for $i>0$. 
\end{prop}

\begin{proof} The result for a flasque $F$ is well known. If $F$ is coflasque the result follows from the following lemmas:
\begin{lem}\label{coflasque} Let $0\to F'\to F\to F''\to 0$ be an exact sequence of abelian data on $X$. If $F''$ is coflasque, then the sequence
\[ 0\to L(X,F')\to L(X,F)\to L(X,F'')\to 0\] is exact.
\end{lem}

\begin{proof} If $X$ is irreducible, then $X=C_p$, so $L(X,F)=F_p$ for any $F$ and the result is immediate. If $X$ is not irreducible, then $X$ is the union of two proper closed subsets $X=X_1\cup X_2$. By induction on the number of points, the result holds on $X_1,X_2$ and $X_{12}=X_1\cap X_2$, so we obtain a commutative diagram of exact sequences


\[ \xymatrix{   0 \to L(X_1,F')\oplus L(X_2,F') \ar[r] & L(X_1,F)\oplus L(X_2,F)\ar[r] &  L(X_1,F'')\oplus L(X_2,F'') \to 0\\ 0 \longrightarrow L(X_{12},F')\ar[r] \ar[u] & L(X_{12},F)\ar[r] \ar[u] & L(X_{12},F'') \longrightarrow 0  \ar[u] 
}\] 
One concludes the result by the snake lemma, since $L(X_{12},F'')\to L(X_{i},F'')$ is injective, because $F''$ is coflasque.\end{proof}

\begin{lem} If $P$ is projective then it is coflasque.
\end{lem}

\begin{proof} Let $Z$ be a closed subset and $G$ an abelian group. One has an epimorphism $G\to G_Z$, hence an epimorphism $\Hom(P,G)\to\Hom(P,G_Z)$; by Proposition \ref{propsop1}, this yields an  epimorphism
\[\Hom(L(X,P),G)\to \Hom(L(Z,P),G).\] Since this holds for any $G$, $L(Z,P)\to L(X,P)$ is injective (in fact, a direct summand).
\end{proof}

\begin{lem} Let $0\to F'\to F\to F''\to 0$ be an exact sequence of abelian data on $X$. If $F$ and $F''$ are coflasque, then $F'$ is coflasque too.
\end{lem}
\begin{proof} It follows easily from Lemma \ref{coflasque}.
\end{proof}
\end{proof}

\subsection{Standard resolutions:}  Finite resolutions by flasque and coflasque abelian data.
\medskip

Let $X$ be a finite  space. A  complex $F$   of abelian data on $X$ may be denoted as a sequence
\[ \cdots\to F^n  \to F^{n+1}\to \cdots\] whose differential has degree $+1$, or as a sequence 
\[ \cdots\to F_n  \to F_{n-1}\to \cdots\] whose differential  has degree $-1$. The equivalence between both presentations is given by $F^{-n}=F_n$.

\begin{defn} Let $F$ be an abelian data on $X$. For each integer $n\geq 0$,   $C^nF$ is the sheaf on $X$ defined by:

\[ (C^nF)(U)=\prod_{(x_0<\cdots <x_n)\in U}F_{x_n}. \]  
\end{defn}
These groups form a complex 
\[ 0\to (C^0F)(U)\to (C^1F)(U)\to \cdots \to (C^dF)(U)\to 0,\quad d=\dim X.\] hence we have a complex $C^\punto F$ of abelian data
\[ 0\to C^0F\to C^1F\to \cdots \to C^dF\to 0\] and an augmentation $F\to C^\punto F$. 
\begin{thm}\label{finiteflasqueresolution} $C^\punto F$ is a finite and flasque resolution of $F$. Hence, for any open subset $U\subseteq X$: $$H^i[\Gamma(U,C^\punto F)]  \simeq H^i(U,F).$$ 
\end{thm}

The analogous construction for homology is the following. 

\begin{defn} For each integer $n\geq 0$,  $C_nF$ is the cosheaf on $X$ defined by 
\[ (C_nF)(C)=\bigoplus_{(x_0<\cdots <x_n)\in C}F_{x_0}. \] 
\end{defn}
We have a complex 
\[ 0\to (C_dF)(C)\to  \cdots \to (C_1F)(C)\to (C_0F)(C)\to 0,\quad d=\dim X\] hence we have a complex $C_\punto F$ of abelian data
\[ 0\to C_dF\to \cdots \to C_1F\to  C_0F\to 0\] and an augmentation $  C_0 F\to F$.

\begin{thm}\label{finitecoflasqueresolution} $C_\punto F$ is a finite and coflasque resolution of $F$. Hence, for any closed subset $Z\subseteq X$:  $$H_i[L(Z,C_\punto F) ]\simeq H_i(Z,F).$$
\end{thm}

\begin{cor}\label{finiteness} Let $F$ be an abelian data on $X$. If $F$ is of finite type (i.e., $F_p$ is a finitely generated abelian group for any $p\in X$), then $H^i(X,F)$ and $H_i(X,F)$ are finitely generated.
\end{cor}

\begin{proof} By hypothesis, $\Gamma(X,C^\punto F) $ (resp. $L(X,C_\punto F )$) is a complex of finitely generated abelian groups. On concludes by Theorem \ref{finiteflasqueresolution} (resp. Theorem \ref{finitecoflasqueresolution}).
\end{proof}

\begin{rem}\label{C^n&C_n} Let $X_\text{dis}$ be the set $X$ with the discrete topology, and  $\id\colon X_\text{dis} \to X$ the identity map. Then $C^0F=\id_*\id^{-1}F$ and $C_0F=\id_!\id^{-1}F$.  
More generally, for each $n>0$, let $X_<^n$ be the subspace of $X\times\overset n\cdots \times X$ defined as $$X_<^n=\{ (x_1,\dots,x_n)\in X^n: x_1<\cdots <x_n\}$$ and let $\pi_1\colon X_<^n\to X$, $\pi_n\colon X_<^n\to X$ be the projections on the first and last factors. Then
\[ C^nF={\pi_1}_*C^0(\pi_n^{-1}F),\qquad C_nF={\pi_n}_!C_0(\pi_1^{-1}F).\]
\end{rem}

\section{Derived Category}

 We shall denote by $C(X)$ the category of complexes of abelian data on $X$ and by $D(X)$ its derived category.

\subsection{Standard results}
\medskip

In this subsection we give standard concepts and results regarding derived categories and derived functors on a finite space. 

For any $F,F'\in C(X)$ we shall denote by $\Hom^\punto(F,F')$ the complex of homomorphisms. A complex $I$ is called {\it $K$-injective} if $\Hom^\punto(\quad, I)$ takes acyclic complexes into acyclic complexes (equivalently, quasi-isomorphisms into quasi-isomorphisms). A complex $P$ is called {\it $K$-projective} if $\Hom^\punto(P,\quad)$ takes acyclic complexes into acyclic complexes (equivalently, quasi-isomorphisms into quasi-isomorphisms).

\begin{thm} The category $C(X)$ has enough $K$-injectives and $K$-projectives: for each $F\in C(X)$ there is a functorial quasi-isomorphism $F\to I(F)$ (resp. $P(F)\to F$) with $I(F)$ a $K$-injective complex (resp. $P(F)$ a $K$-projective complex).
\end{thm}

\subsubsection{} The right derived functor  $\RR\Hom^\punto(F,F')$ is by definition: $$\RR\Hom^\punto(F,F'):=  \Hom^\punto(F,I(F')).$$  We can also derive the complex of homomorphisms by the left: $\LL\Hom^\punto(F,F')= \Hom^\punto(P(F),F')$; both are isomorphic, since they are both isomorphic to $ \Hom^\punto  (P(F),I(F'))$. 

\subsubsection{} For each open subset $U$ of $X$, the right derived functor of the functor $\Gamma(U,\quad)$ is    $$\RR\Gamma (U,F):=\Gamma(U,I(F)).$$  For each closed subset $C$ of $X$, the left derived functor of the functor $L(C,\quad)$ is    $$\LL L  (C,F):=L(C,P(F)).$$  For simplicity, we shall denote $$\LL(C,F):=\LL L(C,F).$$ 
More generally, if $f\colon X\to Y$ is a continous map, we have the functors $$\RR f_*F:=f_*I(F) \quad \text{\rm and}\quad \LL f_!F:=f_!P(F).$$  If $Y$ is a point, we recover the functors $\RR\Gamma(X,F)$ and $\LL(X,F)$. 
The stalkwise description of $\RR f_*F$ and $\LL f_!F$ is:
\[ (\RR f_*F)_y=\RR\Gamma(f^{-1}(U_y), F),\qquad (\LL f_!F)_y=\LL(f^{-1}(C_y), F).\]

If $X\overset f\to Y\overset g\to Z$ are continuous maps one has natural isomorphisms
\[ \RR(g\circ f)_*F\simeq \RR g_*\RR   f_*F \quad \text{\rm and}\quad   \LL(g\circ f)_!F\simeq \LL g_!\LL   f_!F .\]

These functors may be calculated with the standard resolutions: for each complex $F\in C(X)$, we shall denote $C^\punto F$ and $C_\punto F$ the simple complexes associated to the bicomplexes $C^pF^q$ and $C_pF_q$ respectively; then, one has isomorphisms $$\RR f_*F\simeq f_*C^\punto(F)\quad \text{\rm and}\quad \LL f_!F\simeq f_!C_\punto(F).$$  

Direct image commutes with filtered direct limits. Since $C^0F=\id_*\id^{-1}F$, one has that $C^0$ commutes with filtered direct limits. Now, since $C^n F= {\pi_0}_* C^0(\pi_n^{-1} F)$ (Remark \ref{C^n&C_n}), one has that $C^n$ also commutes with filtered direct limits. Thus, if $f\colon X\to Y$ is a continuous map and  $\{ F_i\}$ a filtered direct system of complexes of abelian data on $X$, we have an isomorphism $\ilim{} f_*C^\punto F_i \simeq f_*C^\punto (\ilim{} F_i) $ and 
\[ \ilim{i} \RR f_* F_i\to \RR f_*(\ilim{i}F_i)\] is an isomorphism in the derived category.

\begin{defn} A complex $F\in D(X)$ is said to be of {\it finite type} if $H^i(F_x)$ is a finitely generated abelian group, for any $i\in \ZZ$ and any $x\in X$.
\end{defn}
  Corollary \ref{finiteness}  can be generalized in a standard way to the following

\begin{thm} [Finiteness Theorem] Let $f\colon X\to Y$ be a continous map and $F\in D(X)$. If $F$ is of finite type, then $\RR f_* F$ and $\LL f_! F$ are also of   finite type.
\end{thm}


\subsubsection{} For any $F,F'\in C(X)$, the  complex of sheaves of homomorphisms $\HHom^\punto (F,F')$ is defined by $$\Gamma(U,\HHom^\punto(F,F')):=\Hom^\punto(F_{\vert U},F'_{\vert U}).$$ Thus, for each point $x\in X$,   $\HHom^\punto (F,F')_x=\Hom^\punto(F_{\vert U_x},F'_{\vert U_x})$. The right derived functor is $$\RR\HHom^\punto(F,F'):=\HHom^\punto(F,I(F')).$$

For any open subset $U$ of $X$ one has canonical isomorphisms
\[\RR\Gamma(U,\RR\HHom^\punto(F,F'))\simeq \RR\Hom^\punto(F_{\vert U}, F'_{\vert U}),\qquad \RR\HHom^\punto(F,F')_{\vert U}\simeq  \RR\HHom^\punto(F_{\vert U},F'_{\vert U})\]

\begin{defn} For any $F\in D(X)$, the {\it dual} $F^\vee$ is defined as:
\[ F^\vee:=\RR\HHom^\punto (F,\ZZ).\]
\end{defn}

\subsubsection{} For each $F,F'\in C(X)$ we shall denote $F\otimes F'$ the tensor product complex. Its left derived functor is
\[ F\overset \LL\otimes F':= P(F)\otimes F'.\]

\subsubsection{Locally constant complexes. Perfect complexes}

\begin{defn} A complex $\Lc\in C(X)$ is called {\it locally constant} (resp. {\it constant}) if it is a complex of locally constant (resp. constant) abelian data on $X$. 
\end{defn}

A complex $\Lc\in C(X)$ is locally constant if and only if the morphisms $r_{pq}\colon \Lc_p\to\Lc_q$ are isomorphisms, for any $p\leq q$. We shall denote by $C(\text{\rm Qcoh}(X))$ the category of locally constant complexes, which is a full subcategory of $C(X)$, and by $D(\text{\rm Qcoh}(X))$ its derived category. 

Let $\pi\colon X\to \{*\}$. If $X$ is connected, the inverse image $\pi^{-1}\colon C(\{*\})\to C(X)$ establishes an equivalence between the category of complexes of abelian groups and the category of constant complexes on $X$.

We shall denote by $D_{\text{qc}}(X)$ (resp. $D_\text{c}(X)$) the full subcategory of $D(X)$ consisting of the complexes with locally constant (resp. locally constant of finite type) cohomology. The objects of $D_{\text{qc}}(X)$  are characterized by the following property: for any $p\leq q$ the morphism $r_{pq}\colon \Lc_p\to\Lc_q$ is a quasi-isomorphism. 

One has  natural functors $D(\text{\rm Qcoh}(X))\to D_{\text{qc}}(X)$ and $D(\text{\rm Coh}(X))\to D_{\text{c}}(X)$ which are not equivalences in general. However, they are equivalences if $X=U_p$ or $X=C_p$ (i.e., $X$ has a minimum or a maximum).

\begin{defn} A complex in $ D(X)$ is called {\it perfect} if it is locally isomorphic (in the derived category) to a bounded complex of free $\ZZ$-modules of finite type.
\end{defn}

We shall denote $D_\text{perf}(X)$ the full subcategory of $D(X)$ consisting in perfect complexes. One has a natural inclusion $D_\text{perf}(X)\hookrightarrow D_c^b(X)$, where $D^b_c(X)$ is the full subcategory of complexes in $D_c(X)$ with bounded cohomology. This inclusion  is an equivalence; this is essentially due to the fact that any finitely generated abelian group admits a finite resolution (in fact a resolution of length one) by finitely generated free abelian groups. If $F\in D(X)$ is of finite type and has bounded cohomology, then $\LL(X,F)$ and $\RR\Gamma(X,F)$ are perfect complexes of abelian groups.

If $\Lc\in D(X)$ is perfect, the natural morphism $\Lc\to \Lc^{\vee\vee}$ is an isomorphism. Moreover, for any $F\in D(X)$ one has an isomorphism 
\[ \Lc^\vee\overset\LL\otimes F\overset\sim\to  \RR\HHom^\punto(\Lc,F)\] and then, for any $F,K\in D(X)$, an isomorphism  
\[ \RR\Hom^\punto(K,  \Lc^\vee\overset\LL\otimes F)\overset\sim\to  \RR\Hom^\punto(K\overset\LL\otimes\Lc,F) \]


\subsubsection{Adjunctions}

\begin{prop}[Adjunctions $\LL f_!$ $\leftrightarrow$ $f^{-1}$ $\leftrightarrow$ $\RR f_*$]\label{AdjunctionsOfInvIm} Let $f\colon X\to Y$ be a continuous map. For any $F\in D(X)$, $G\in D(Y)$ one has
\[ \aligned \RR\Hom^\punto (f^{-1}G,F)& = \RR\Hom^\punto (G,\RR f_*F),\\  \RR\Hom^\punto (F,f^{-1}G) & = \RR\Hom^\punto ( \LL f_!F,G).\endaligned\]
\end{prop}

\begin{proof} We shall only give the prove of the second part. Let $P\to F$ be a projective resolution. Then
\[ \RR\Hom^\punto (F,f^{-1}G) =  \Hom^\punto (P,f^{-1}G)= \Hom^\punto (f_!P , G) =\RR\Hom^\punto(\LL f_!F, G)\] where  the last equality is due to the projectivity of $f_!P$.
\end{proof}

\begin{prop}[Adjunction $\overset\LL\otimes$ $\leftrightarrow$ $\RR\HHom$]\label {RHom&Tens} For any $F_1,F_2,F_3\in D(X)$ one has a natural isomorphism
\[ \RR\Hom^\punto (F_1,\RR\HHom^\punto(F_2,F_3)) \simeq \RR\Hom^\punto (F_1\overset\LL\otimes F_2,F_3).\] 
\end{prop}

\begin{prop}\label{InvImTensHom} Let $f\colon X\to Y$ be a continuous map and let $G,\Lc\in D(Y)$. 
\begin{enumerate} \item There is a natural isomorphism
\[ f^{-1}(G\overset\LL\otimes \Lc)\simeq (f^{-1}G)\overset\LL\otimes (f^{-1} \Lc).\] 
\item There is a natural morphism  \[ f^{-1} \RR\HHom^\punto(\Lc ,G)\to  \RR\HHom^\punto(f^{-1}\Lc , f^{-1} G)\] which is an isomorphism if $\Lc$ is locally constant.
\end{enumerate}
\end{prop}

\begin{proof} We shall only prove the last statment of (2). 
Assume that $\Lc$ is locally constant. Let $x\in X$ and $y=f(x)$. Then 
\[ (f^{-1}\RR\HHom^\punto (\Lc,K))_x=  \RR\HHom^\punto (\Lc,K)_y= \RR\Hom^\punto(\Lc_{\vert U_y},K_{\vert U_y} )=  \RR\Hom^\punto(\Lc_y,K_y ),\] where the last equality is due to Lemma \ref{lemita} ($\Lc_{\vert U_y}$ is constant). On the other hand 
\[\aligned \RR\HHom^\punto (f^{-1}\Lc, f^{-1}K)_x =\RR\Hom^\punto ((f^{-1} \Lc)_{\vert U_x}, (f^{-1} K)_{\vert U_x})&= \RR\Hom^\punto ((f^{-1} \Lc)_x, (f^{-1} K)_x)\\ &= \RR\Hom^\punto(\Lc_y,K_y ), \endaligned
\] where the second equality is due again to Lemma \ref{lemita} (since $(f^{-1} \Lc)_{\vert U_x}$ is constant)

\end{proof}

\begin{lem}\label{lemita} Let $X$ be   connected. For any  $K\in D(X)$ and any constant $G\in D(X)$, one has
\[\RR\Hom^\punto(G,K)=\RR\Hom(G,\RR\Gamma(X,K))\]
\end{lem}

\begin{proof} Let $\pi\colon X\to \{*\}$. Since $G$ is constant, $G=\pi^{-1}G$. One concludes by adjunction between $\pi^{-1}$ and $\RR\pi_*$.
\end{proof}

\section{Base Change and (Co)Projection Formula}

In this section we establish  projection formulae morphisms and base change morphisms for $\RR f_*$ and $\LL f_!$. We shall introduce the notions of c-proper and h-open maps and see their relation with the validity of base change theorems and projection formulae.

\subsection{Base change for cohomology and homology}

\begin{prop}[Base Change for homology and cohomology]\label{basechangemorphism} Let \[ \xymatrix{   \overline Y\times_YX \ar[r]^{\quad\bar g}\ar[d]_{\bar f} & X\ar[d]^f  \\ \overline Y \ar[r]^g   & Y
}\] be a cartesian diagram and let $F\in D(X)$. One has   natural morphisms
\[ g^{-1}\RR f_*F \to \RR {\bar f}_* ({\bar g}^{-1}F) \]
\[ \LL {\bar f}_! ({\bar g}^{-1}F)  \to g^{-1}\LL f_!F.\]
If $g$ is an open immersion (resp. a closed immersion), then  $g^{-1}\RR f_*F \to \RR {\bar f}_* ({\bar g}^{-1}F)$ is an isomorphism (resp. $g^{-1}\LL f_!F \leftarrow \LL {\bar f}_! ({\bar g}^{-1}F)$ is an isomorphism).
\end{prop}

\begin{proof} Since the cohomological statement is well known, we prove the homological one. Let $P\to F$ and $Q\to {\bar g}^{-1}P$ be projective resolutions. By Proposition \ref{1.2.11}, one has a morphism
\[ {\bar f}_!Q\to {\bar f}_!({\bar g}^{-1} P)\overset{\ref{1.2.11}}\to g^{-1} f_! P\]  that gives the desired morphism $ \LL {\bar f}_! ({\bar g}^{-1}F) \to g^{-1}\LL f_!F $.

Now assume that  $g$ is a closed immersion and the diagram is cartesian. Then  $g^{-1}\circ f_! = {\bar f}_!\circ {\bar g}^{-1}$ (Proposition \ref{1.2.10}) and ${\bar g}^{-1}P$ is a projective resolution of ${\bar g}^{-1} F$, because $\bar g$ is a closed immersion; hence $ \LL {\bar f}_! ({\bar g}^{-1}F) \to g^{-1}\LL f_!F $ is an isomorphism. 
\end{proof}
%

\begin{rem}
If $\overline Y=\{ y\}$ is just a point of $Y$, then the base change morphisms are
\[ (\RR f_*F)_y \to \RR\Gamma( f^{-1}(y),F_{\vert f^{-1}(y)})  \]
\[ \LL (f^{-1}(y),F_{\vert f^{-1}(y)})\to (\LL f_!F)_y  .\] In particular, for any abelian data $F$ on $X$, one has the more ``classic'' morphisms 
\[ (R^i f_*F)_y \to H^i( f^{-1}(y),F_{\vert f^{-1}(y)})  \]
\[ H_i (f^{-1}(y),F_{\vert f^{-1}(y)})\to (L_i f_!F)_y   .\] 
\end{rem}

%

\subsection{c-proper and h-open morphisms}

In the context of locally compact and separated spaces, the base change morphisms for cohomology are isomorphisms, provided that $f$ is proper (see \cite{Iversen}). However, for finite spaces, properness (in the sense of universally closed) is necessary but not sufficient to have base change isomorphisms. For finite spaces, the notion of a proper map is too weak (it is equivalent to being a closed map, Proposition \ref{propermaps}) and a stronger notion is necessary. The necessary notion is that of a cohomologically proper map (see Definition \ref{c-proper}). Analogously, in order to obtain a base change theorem for homology we need the notion of homologically open map, which is a stronger notion than that of an universally open map. 

\begin{defn} A finite space $X$ is called  {\it homologically trivial} if it satisfies any of the equivalent conditions:
\begin{enumerate}
\item $H^i(X,\ZZ)=0$ for any $i>0$ and $H^0(X,\ZZ)=\ZZ$.
\item $H^i(X,G)=0$ for any $i>0$ and $H^0(X,G)=G$ for any abelian group $G$.
\item $H_i(X,\ZZ)=0$ for any $i>0$ and $H_0(X,\ZZ)=\ZZ$.
\item $H_i(X,G)=0$ for any $i>0$ and $H_0(X,G)=G$ for any abelian group $G$.
\end{enumerate}
\end{defn}

\begin{rem} The equivalence of conditions (1)-(4) may be deduced from the following facts. On the one hand, one has that $\Gamma (X,C^\punto\ZZ)\otimes_\ZZ G = \Gamma (X,C^\punto G)$ and $L(X,C_\punto\ZZ)\otimes_\ZZ G = L(X,C_\punto G)$; it follows that (1) is equivalent to (2) and (3) is equivalent to (4). On the other hand, one has $\LL(X,\ZZ)^\vee=\RR\Gamma (X,\ZZ)$ (see Corollary \ref{Cor-homo-cohomo}) and $\LL(X,\ZZ)\simeq \LL(X,\ZZ)^{\vee\vee}=\RR\Gamma (X,\ZZ)^\vee$ (where the first isomorphism is due to the perfectness of $\LL(X,\ZZ)$) and then (1) is equivalent to (3).
\end{rem}

Notice that any homologically trivial space is connected, because $H^0(X,\ZZ)=\ZZ$. If $X$ is contractible (more generally, if $X$ is homotopically trivial), then it is homologically trivial.

\begin{prop}\label{propermaps} Let $f\colon X\to Y$ be a continuous map. The following conditions are equivalent.
\begin{enumerate}
\item $f$ is universally closed.
\item $f$ is closed.
\item For any $x\in X$, the map $f_{\vert C_x}\colon C_x\to C_{f(x)}$ is surjective.
\end{enumerate}
\end{prop}

\begin{proof} (1) $\Rightarrow$ (2) $\Leftrightarrow$ (3) are immediate. Now assume that (3) holds and let $\overline Y\to Y$ be a continuous map and $ \bar f\colon\overline Y \times_YX\to \overline Y$ the natural map. For any $(\bar y,x)\in \overline Y \times_YX$ one has $C_{(\bar y,x )}=C_{\bar y}\times_{C_y}C_x$, where $y$ is the image of $x$ and $\bar y$ in $Y$. Since $C_x\to C_y$ is surjective, $C_{\bar y}\times_{C_y}C_x\to C_{\bar y}$ is also surjective; hence $\bar f$ is closed.
\end{proof}

\begin{prop} Let $f\colon X\to Y$ be a continuous map. The following conditions are equivalent.
\begin{enumerate}
\item $f$ is universally open.
\item $f$ is open.
\item For any $x\in X$, the map $f_{\vert U_x}\colon U_x\to U_{f(x)}$ is surjective.
\end{enumerate}
\end{prop}

\begin{proof} Apply the preceding Proposition to $\widehat f\colon \widehat X\to\widehat Y$.
\end{proof}

\begin{defn}\label{c-proper} A continuous map $f\colon X\to Y$ is called {\it cohomologically proper} (c-proper for short) if:
\begin{enumerate}
\item $f$ is closed.
\item For any $x\in X$, the induced map $f_{\vert C_x}\colon C_x\to C_{f(x)}$ has homologically trivial fibers: for any $y\in C_{f(x)}$, $f_{\vert C_x}^{-1}(y)$ is homologically trivial.
\end{enumerate}
\end{defn}

\begin{defn} A continuous map $f\colon X\to Y$ is called {\it homologically open} (h-open for short) if:
\begin{enumerate}
\item $f$ is open.
\item For any $x\in X$, the induced map $f_{\vert U_x}\colon U_x\to U_{f(x)}$ has homologically trivial fibers: for any $y\in U_{f(x)}$, $f_{\vert U_x}^{-1}(y)$ is homologically trivial.
\end{enumerate}
\end{defn}

Any closed (resp. open) immersion is c-proper (resp. h-open). For any $X$, the map $X\to \{*\}$ is c-proper and h-open. If $f\colon X\to Y$ is c-proper (resp. h-open), then $f_{\vert C}\colon C\to Y$ (resp. $f_{\vert U}$) is c-proper (resp. h-open) for any closed subset $C$ of $X$ (resp. any open subset $U$).  A continuous map $f\colon X\to Y$ is c-proper if and only if $\widehat f\colon\widehat X\to\widehat Y$ is h-open. Finally, c-properness and h-openness remain after base change:

\begin{prop}\label{c-prop-cambia} Let $f\colon X\to Y$ be a c-proper map (resp. an h-open map). For any continuous map $ \overline Y\to Y$, the induced map $\bar f\colon \overline Y\times_Y X\to \overline Y$ is c-proper (resp. h-open).
\end{prop}

\begin{proof} $\bar f$ is closed by Proposition \ref{propermaps}. Now, let $(\bar y, x)\in \overline Y\times_Y X$ and $y$ the image of $x$ and $\bar y$ in $Y$. For any $\bar y_1\in C_{\bar y}$, the fiber of $\bar y_1$ by the map $C_{(\bar y,x)}=C_{\bar y}\times_{C_y}C_{x}\to C_{\bar y}$ is homeomorphic to the fiber of $y_1$ (= image of $\bar y_1$ in $Y$) by the map $ C_x\to C_y$. Conclusion follows.
\end{proof}

\begin{thm}\label{c-proper&basechange} Let $f\colon X\to Y$ be a continuous map. The following conditions are equivalent:
\begin{enumerate}
\item $f$ is c-proper.
\item $f$ satisfies the base change theorem for cohomology: For any abelian data $F$ on $X$ and any $y\in Y$, the base change morphism
\[ (R^if_*F)_y\to H^i(f^{-1}(y),F_{\vert f^{-1}(y)})\] is an isomorphism.
\end{enumerate}
\end{thm}

\begin{proof} Assume that $f$ is c-proper. Let us prove that  $f$ satisfies the base change theorem. We proceed by induction on $n$= the number of points of $X$. For $n=1$ it is immediate (notice that $f$ is closed). If $X$ is the union of two proper closed subsets $C_1$ and $C_2$, let us consider the exact sequence (where $C_{12}=C_1\cap C_2$)
\[ 0\to F\to F_{C_1}\oplus F_{C_2}\to F_{C_1\cap C_2}\to 0.\] Let us denote $f_1$ (resp. $f_2$, $f_{12}$) the restriction of $f$ to $C_1$ (resp. to $C_2$, $C_{12}$) and $X^y=f^{-1}(y)$, $C_1^y=f_1^{-1}(y)$ and analogously $C_2^y$ and $C_{12}^y$. One has by induction
\[ (R^if_*F_{C_1})_y= (R^i{f_1}_*F_{\vert C_1})_y \overset\sim\to H^i(C_1^y,F_{\vert C_1^y}) = H^i(X^y,(F_{\vert X^y})_{C_1^y})\] 
and analogously for $f_2$ and $f_{12}$. One concludes by the commutative diagram of exact sequences
\[ \xymatrix{\cdots \ar[r] & (R^if_*F)_y\ar[r]\ar[d] & (R^if_*F_{C_1})_y \oplus (R^if_*F_{C_2})_y\ar[r] \ar[d]^{\wr} & (R^if_*F_{C_{12}})_y\ar[r]\ar[d]^{\wr} & \cdots \\ 
 \cdots \ar[r] & H^i(X^y,F_{\vert X^y})\ar[r] & H^i(X^y,(F_{\vert X^y})_{C_1^y}) \oplus H^i(X^y,(F_{\vert X^y})_{C_2^y})\ar[r] & H^i(X^y,(F_{\vert X^y})_{C_{12}^y})\ar[r] & \cdots }\] Thus, we may assume that $X$ is irreducible. Let $g$ be the generic point of $X$ and $C=X-\{ g\}$. Let us consider the exact sequence 
 \[ 0\to F_{\{g\}}\to F\to  F_C \to 0.\] By induction, $F_C$ satisfies base change; hence it suffices to prove that  $F_{\{g\}}$ satisfies base change. Let $G$ be the stalk of $F$ at $g$ and let us still denote by $G$ the constant sheaf on $X$. Notice that $G_{\{g\}}= F_{\{g\}}$, so one has the exact sequence
\[ 0\to F_{\{g\}}\to G \to  G_C \to 0\] and $G_C$ satisfies   base change by induction. We are reduced to prove the statement for the constant sheaf $G $, ie, to prove that $(R^if_*G)_y\to H^i(f^{-1}(y),G)$ is an isomorphism. If $y$ is not in the image of $f$, both members are zero (notice that $f$ is closed), hence we may assume that $y\in\Ima f$. Now,   $(R^if_*G)_y=H^i(f^{-1}(U_y), G)=0$ for $i>0$ because $f^{-1}(U_y)$ is contractible to $g$, and $(f_*G)_y=G$. On the other hand, $H^i(f^{-1}(y), G)=0$ for $i>0$ and $H^0(f^{-1}(y),G)=G$ because $f^{-1}(y)$ is homologically trivial.

Assume now that $f$ satisfies the base change theorem. Let us first see that $f$ is closed. Let $y\in Y$ be an element of the closure of $\Ima f$. Then $f^{-1}(U_y)$ is not empty, hence $(f_*\ZZ)_y\neq 0$. By hypothesis, $(f_*\ZZ)_y\simeq \Gamma(f^{-1}(y),\ZZ)$, hence $f^{-1}(y)$ is not empty and then $y\in\Ima f$.

Let $x\in X$ and let us prove that  $f_{\vert C_x}\colon C_x\to C_{f(x)}$ has homologically trivial fibers. Notice that $f_{\vert C_x}$ is surjective (because $f$ is closed) and it also satisfies the base change theorem. Hence, for any $y\in C_{f(x)}$ one has:
\[ H^i(f_{\vert C_x}^{-1}(U_y),G)\overset\sim\to H^i(f_{\vert C_x}^{-1}(y),G_{\vert f_{\vert C_x}^{-1}(y)}).\] If $G$ is constant, then $H^i(f_{\vert C_x}^{-1}(U_y),G)=0$ for any $i>0$ and $H^0(f_{\vert C_x}^{-1}(U_y),G)=G$, because $f_{\vert C_x}^{-1}(U_y)$ is contractible ($x$ is a maximum). Conclusion follows.
\end{proof}
A completely analogous proof gives the following:
\begin{thm}\label{h-open&basechange} Let $f\colon X\to Y$ be a continuous map.  The following conditions are equivalent:
\begin{enumerate}
\item $f$ is h-open.
\item $f$ satisfies the base change theorem for homology: For any abelian data $F$ on $X$ and any $y\in Y$, the base change morphism
\[ H_i(f^{-1}(y),F_{\vert f^{-1}(y)})\to  (L_if_!F)_y\] is an isomorphism.
\end{enumerate}
\end{thm}

\begin{cor}\label{3.2.10} (see also \cite[Section 4.]{BjörnerWachsWelker}) Let $f\colon X\to Y$ be an h-open (resp. c-proper) map. If $f$ has homologically trivial fibers (i.e., $f^{-1}(y)$ is homologically trivial for any $y\in Y$), then
\[ H_i(X,G)=H_i(Y,G)\qquad\qquad (\text{\rm resp. } H^i(X,G)=H^i(Y,G)) \] for any $i\geq 0$ and any constant abelian data $G$.
\end{cor}

\begin{proof} We prove the homological statement; the cohomological one is analogous. By Theorem \ref{h-open&basechange}, for any $y\in Y$ one has 
\[ (L_if_!G)_y\overset\sim \leftarrow H_i(f^{-1}(y),G).\] Moreover, $H_i(f^{-1}(y),G)=0$ for $i>0$ and $H_0(f^{-1}(y),G)=G$ because $f^{-1}(y)$ is homologically trivial. Thus, $L_if_!G=0$   for $i>0$ and $G\simeq   f_!G$. Hence $H_i(X,G)=H_i(Y,f_!G)=H_i(Y,G)$.
\end{proof}

\begin{ejem} Let $X$ be a finite space and let $\beta X$ be its barycentric subdivision (i.e., $\beta X$ is the set of properly ascending chains of $X$, and the topology is given by the preorder: $(y_0 <\cdots < y_r)\leq (x_0<\cdots < x_s)$ iff $y_0 <\cdots < y_r$ is a subchain of $x_0<\cdots < x_s$.) One has a continous map
\[ \aligned \pi\colon \beta X&\to X\\ (x_0<\cdots < x_s)&\mapsto x_s.\endaligned\] Then $\pi$ is h-open and has homologically trivial fibers (indeed, contractible fibers). By Corollary \ref{3.2.10}, one obtains the well known fact: \[ H_i(\beta X,G)=H_i(X,G)\] for any abelian group $G$.
\end{ejem}

\begin{cor}\label{cproper&generalbasechange} Let  $f\colon  X\to Y$ be a continuous map. The following conditions are equivalent:
\begin{enumerate} 
\item $f$ is c-proper (resp. h-open).

\item For any $F\in D(X)$ and any  $g\colon \overline Y\to Y$ the base change morphism (Proposition \ref{basechangemorphism})
\[ g^{-1}\RR f_*F \to \RR {\bar f}_* ({\bar g}^{-1}F)  \quad(\text{\rm resp. } \LL {\bar f}_! ({\bar g}^{-1}F)  \to g^{-1}\LL f_!F)\]  is an isomorphism.
\end{enumerate}
\end{cor}
\begin{proof} We shall prove the cohomological statement; the homological one is analogous. If (2) holds, then condition (2) of Theorem \ref{c-proper&basechange} also holds, and then $f$ is c-proper by Theorem \ref{c-proper&basechange}. Conversely, assume that $f$ is c-proper and let $g\colon\overline Y\to Y$ be any continous map. In order to prove that $g^{-1}\RR f_*F \to \RR {\bar f}_* ({\bar g}^{-1}F) $ is an isomorphism, it suffices to prove that it is an isomorphism after taking fiber at any point $\bar y\in\overline Y$ and $H^i$; in this case, the problem is reduced, by standard arguments, to the case that $F$ is an abelian data on $X$, and we have to prove that 
$$ (R^i f_*F)_{g(\bar y)}  \to [R^i {\bar f}_* ({\bar g}^{-1}F)]_{\bar y}$$ 
is an isomorphism. Since $f$ and $\bar f$ are c-proper (Proposition \ref{c-prop-cambia}), one has by Theorem \ref{c-proper&basechange} (let us denote $y=g(\bar y)$):
\[ (R^i f_*F)_y\overset\sim\to H^i(f^{-1}(y), F_{\vert f^{-1}(y)})\]
\[ [R^i {\bar f}_* ({\bar g}^{-1}F)]_{\bar y}\overset\sim\to H^i({\bar f}^{-1}(\bar y), ({\bar g}^{-1}F)_{\vert {\bar f}^{-1}(\bar y)})\] and one concludes because $\bar g\colon {\bar f}^{-1}(\bar y)\to f^{-1}(y)$ is an homeomorphism.
\end{proof}

\begin{cor} The composition of c-proper (resp. h-open) maps is c-proper (resp. h-open).
\end{cor}

\begin{proof} Let $X\overset f\to Y\overset g\to Z$. If $f$ and $g$ are c-proper, then $\RR f_*$ and $\RR g_*$ satisfy base change theorem (Corollary \ref{cproper&generalbasechange}). Since $\RR(g\circ f)_*\simeq \RR g_*\circ\RR f_*$, it follows that $\RR(g\circ f)_*$ also satisfies base change theorem, and then $g\circ f$ is c-proper. For h-open maps the proof is analogous.
\end{proof}

\subsection{Projection and coprojection formulae}

In the context of locally compact and separated topological spaces, projection formula holds for proper maps (\cite{Iversen}). Now, for finite topological spaces, we shall see (Theorem \ref{hsatprojform}) that projection formula (for cohomology) holds if and only if the map is c-proper. Analogously, projection formula for homology holds if and only if the map is h-open.

\begin{prop}[Coprojection Formula]\label{coprojectionformula} Let $f\colon X\to Y$ be a continuous map. For any $F\in D(X)$ and any $\Lc\in D(Y)$  one has  a natural morphism
\[  \LL f_!(F\overset\LL \otimes f^{-1}\Lc) \to (\LL f_!F)\overset\LL\otimes \Lc\]
which is an  isomorphism if $\Lc$ is locally constant.
\end{prop}

\begin{proof} By Proposition \ref{InvImTensHom}, for any $K\in D(Y)$ one has a morphism $$f^{-1}\RR\HHom^\punto (\Lc,K)\to \RR\HHom^\punto (f^{-1}\Lc, f^{-1}K),$$  which is an isomorphism if $\Lc$ is locally constant. Taking $\RR\Hom^\punto(F,\quad)$, one obtains, by adjunctions, a morphism
\[ \RR\Hom^\punto((\LL f_!F)\overset\LL\otimes \Lc,K)\to \RR\Hom^\punto (\LL f_!(F\overset\LL\otimes f^{-1}\Lc), K),\] hence a morphism
\[ \LL f_!(F\overset\LL\otimes f^{-1}\Lc)\to (\LL f_!F)\overset\LL\otimes \Lc,\] which is an isomorphism if $\Lc$ is locally constant.
\end{proof}

\begin{prop}[Projection Formula]\label{projectionformula} Let $f\colon X\to Y$ be a continuous map. For any $F\in D(X)$ and any $\Lc\in D(Y)$ one has a natural morphism
\[ (\RR f_*F)\overset\LL\otimes \Lc \to\RR f_*(F\overset\LL \otimes f^{-1}\Lc)\]
which is an  isomorphism if $\Lc$ is locally constant.
\end{prop}

\begin{proof} The natural morphism $\epsilon\colon f^{-1}\RR f_*F\to F$, induces a morphism $$ f^{-1}(\RR f_*F\overset\LL\otimes \Lc)\overset{\ref{InvImTensHom}}\simeq f^{-1}\RR f_*F\overset\LL\otimes f^{-1} \Lc\overset{\epsilon\otimes 1} \to F\overset\LL\otimes f^{-1} \Lc $$ and, by adjunction, a morphism $(\RR f_*F)\overset\LL\otimes \Lc \to\RR f_*(F\overset\LL \otimes f^{-1}\Lc)$. In order to prove that it is an isomorphism, we may assume that $Y=U_p$. If $\Lc=\ilim{}\Lc_i$, it suffices to prove the statement for each $\Lc_i$, because both sides of the morphism commute with filtered direct limits. Let $\Lc^{\leq n}$ be the complex
\[ \cdots\to \Lc^{n-2}\to \Lc^{n-1}\to Z^n\to 0\] with $Z^n$ the cycles of degree $n$. One has that $\Lc=\ilim{n}\Lc^{\leq n}$, so we may assume that $\Lc$ is bounded above. Since $\Lc$ is locally constant, it is  isomorphic to a bounded above complex of free $\ZZ$-modules. Now let $ \Lc_{\geq -n}$ be the complex
\[ 0\to \Lc^{-n}\to\Lc^{-n+1}\to\Lc^{-n+2}\to\cdots\] One has  $\Lc=\ilim{n}\Lc_{\geq -n}$, so we may assume that $\Lc$ is a bounded complex of free $\ZZ$-modules.  This is easily reduced to the case where $\Lc$ is just a free $\ZZ$-module, which is trivial.
\end{proof}


\begin{thm}\label{hsatprojform} Let $f\colon X\to Y$ be a continuous map. The following conditions are equivalent
\begin{enumerate} 
\item $f$ is h-open (resp. c-proper).
\item Coprojection (resp. projection) formula holds: for any $F\in D(X)$, $\Lc\in D(Y)$,
\[ \LL f_!(F\overset\LL \otimes f^{-1}\Lc)\to  (\LL f_! F)\overset\LL\otimes \Lc \quad (\text{\rm resp. } (\RR f_* F)\overset\LL\otimes \Lc\to \RR f_*(F\overset\LL \otimes f^{-1}\Lc))\] is an isomorphism.

\item $\LL f_!$ (resp. $\RR f_*$) commutes with supporting on closed subsets:  for any closed subset $C\subseteq Y$ and any $F\in D(X)$ the natural morphism
\[  \LL f_! (F_{f^{-1}(C)})\to (\LL f_!F)_C \quad (\text{\rm resp. } (\RR f_*F)_C\to \RR f_* (F_{f^{-1}(C)}))\] is an isomorphism.

\item $\LL f_!$ (resp. $\RR f_*$) commutes with supporting on open   subsets:  for any open   subset $V\subseteq Y$ and any $F\in D(X)$ the natural morphism
\[\LL f_! (F_{f^{-1}(V)}) \to  (\LL f_!F)_V \quad (\text{\rm resp. } (\RR f_*F)_V\to \RR f_* (F_{f^{-1}(V)})) \] is an isomorphism.


\end{enumerate}
\end{thm}
\begin{proof} We give the proofs for the homological statements and leave the reader to do the analogous ones for the respective cohomological statements.

 (2) $\Rightarrow$ (3). It suffices to apply coprojection formula to $\Lc=\ZZ_C$.

(3) $\Rightarrow$ (4). It follows from the commutative diagram of exact triangles
\[\xymatrix{ (\LL f_!F)_{X-C}\ar[r]\ar[d] & \LL f_!F \ar[r]\ar[d]^{\id} & (\LL f_!F)_C\ar[d]^{\wr}\\ 
\LL f_! (F_{f^{-1}(X-C)})\ar[r] &\LL f_! F \ar[r] & \LL f_! (F_{f^{-1}(C)}).
}\]

(4) $\Rightarrow$ (1). Let $V$ be an open subset of $Y$. Let us denote $i\colon V\hookrightarrow Y$ and $j\colon f^{-1}(V)\hookrightarrow X$. Then
\[ i_![(\LL f_!F)_{\vert V}]=(\LL f_!F)_V\overset\sim\to \LL f_! (F_{f^{-1}(V)}) = \LL f_! (j_! F_{\vert f^{-1}(V)})= i_![\LL { f_V}_!(F_{\vert f^{-1}(V)})].\] Restricting to $V$ one concludes that $$(*)\qquad\qquad \qquad\qquad \qquad\qquad  \LL { f_V}_!(F_{\vert f^{-1}(V)})\overset\sim\to (\LL f_!F)_{\vert V}. \qquad\qquad \qquad\qquad \qquad\qquad \qquad\qquad  $$ In particular, let us take $y\in Y$, $V=U_y$ and $F$ an abelian data on $X$. Taking fiber at $y$ and $H_i$ in the isomorphism $(*)$, one obtains (notice that $\{ y\}$ is a closed subset of $U_y$, so  we can apply Proposition \ref{basechangemorphism})
\[ H_i(f^{-1}(y),F_{\vert f^{-1}(y)})\overset\sim\to (L_if_!F)_y.\] By Theorem \ref{h-open&basechange}, $f$ is h-open. 

(1) $\Rightarrow$ (2). Let us prove that
\[\LL f_!(F\overset\LL \otimes f^{-1}\Lc) \to (\LL f_! F)\overset\LL\otimes \Lc \] is an isomorphism. Taking the stalk at a point $y\in Y$ and making use of Corollary \ref{cproper&generalbasechange}, one obtains the morphism
\[  \LL (f^{-1}(y), F_{\vert f^{-1}(y)}\overset\LL\otimes f_y^{-1}\Lc_y)  \to \LL(f^{-1}(y), F_{\vert f^{-1}(y)})\overset\LL\otimes \Lc_y \] with $f_y\colon f^{-1}(y)\to\{ y\}$. This is an isomorphism by coprojection formula (Proposition \ref{coprojectionformula}) applied to the morphism $f_y$.
\end{proof}

%

\section{Dualities}\label{4}

\subsection{Duality between homology and cohomology}

\begin{thm} Let $f\colon X\to Y$ be a continuous map. For any    $F\in D(X)$  and $G\in D(Y)$ one has:
\[\RR\Hom^\punto (\LL f_!F, G )=  \RR\Gamma(X,\RR\HHom^\punto (F,f^{-1}G)).\]
In particular,
\[\RR\Hom^\punto (\LL f_!\ZZ, G )= \RR\Gamma(X, f^{-1}G).\] 
\end{thm}

\begin{proof} $\RR\Hom^\punto (\LL f_!F, G )= \RR\Hom^\punto (F,f^{-1}G)=\RR\Gamma(X,\RR\HHom^\punto (F,f^{-1}G)).$
\end{proof}

\noindent{\bf\ \ Notation.} Recall that we are using the simplified notation $$\LL (X,F) =\LL\OOgamma (X,F)=\OOgamma (X,C_\punto F).$$

\begin{cor}\label{Cor-homo-cohomo} For any  $F\in D(X)$ and any complex $G$ of abelian groups  one has $$\RR\Hom^\punto (\LL (X,F),G)=\RR\Gamma (X,\RR\HHom^\punto(F,G)).$$
In particular, 
\begin{enumerate} \item $\LL (X,F)^\vee = \RR\Gamma(X,F^\vee)$.
\item For any closed subset $Y\subseteq X$, $$\RR\Hom^\punto (\LL (X,\ZZ_Y),G)=\RR\Gamma_Y (X, G ).$$
\end{enumerate}
\end{cor}

\subsection{Duality between $X$ and $\widehat X$}\label{dualspace}$\,$\medskip

\begin{defn} Let $F$ be an abelian data on $X$ and $\Lc$ a locally constant data on $X$. For each $\wh p\in \wh X$, we shall denote 
\[ \DHom(F,\Lc)_{\widehat p}=\Hom(F_p,\Lc_p).\] For each $\wh p\leq \wh q$, one has a morphism $F_q\to F_p$ and an isomorphism $\Lc_q\overset\sim\to\Lc_p$ ($\Lc$ is locally constant) that induce a   morphism
\[ \DHom(F,\Lc)_{\wh p}=\Hom(F_p,\Lc_p)\to \Hom(F_q,\Lc_p)\overset\sim\leftarrow \Hom(F_q,\Lc_q)=\DHom(F,\Lc)_{\wh q}.\] Thus
$\DHom(F,\Lc)$ is an abelian data on $\widehat X $. 
\end{defn} 


We can derive the construction $\DHom(F,\Lc)$. Let $F$ be a complex of abelian data and $\Lc$ a complex of locally constant abelian data. We define $\DHom^\punto(F,\Lc)$ as: $$\DHom^\punto(F,\Lc)_p=\Hom^\punto(F_p,\Lc_p)$$ and we shall   denote
\[ \LL\DHom^\punto(F,\Lc):=\DHom^\punto(P,\Lc)\] where $P\to F$ is a projective resolution. This defines a contravariant functor
\[ \aligned D(X)&\to D (\widehat X)\\ F &\mapsto \LL\DHom^\punto(F,\Lc)\endaligned.\]
Finally, we shall denote $\widehat F:=\LL\DHom^\punto(F,\ZZ)$.

\begin{prop}\label{puntuallyprojective} Let $F\in C(X)$. If $F_x$ is projective for any $x\in X$, then the natural morphism 
\[ \DHom^\punto(F,\Lc)\to\LL\DHom (F,\Lc)\] is a quasi-isomorphism (an isomorphism in $D(\widehat X)$).
\end{prop}
\begin{proof}  Let $P\to F$ be a projective resolution. For any $\wh x\in \wh X$ the morphism
\[ \Hom^\punto(F_x,\Lc_x)=\DHom^\punto(F,\Lc)_{\wh x}\to\LL\DHom (F,\Lc)_{\wh x}=\Hom^\punto(P_x,\Lc_x)\] is a quasi-isomorphism because $P_x$ is a projective resolution of $ F_x$ and $F_x$ is projective.
\end{proof}

\begin{rem} Let $\Lc$ be a locally constant abelian data on $X$. For any $p\leq q$, $r_{pq}\colon\Lc_p\to\Lc_q$ is an isomorphism. Thus $\Lc$ may be viewed as a locally constant abelian data on $\widehat X$: $$\Lc_{\widehat p}:=\Lc_p \quad\text{ and }\quad r_{\widehat p\widehat q}=r_{qp}^{-1}.$$ It is immediate to see that this abelian data is $\DHom^\punto(\ZZ,\Lc)$ (moreover, $\DHom^\punto(\ZZ,\Lc)\to\LL\DHom^\punto(\ZZ,\Lc)$ is an isomorphism in the derived category, by Proposition \ref{puntuallyprojective}).  Thus, we obtain an equivalence
\[\aligned  \left\{\aligned &\text{Locally constant}\\ &\text{abelian data on }X\endaligned\right\}&\overset\sim  \longrightarrow \left\{\aligned &\text{Locally constant}\\ &\text{abelian data on }\widehat X\endaligned\right\} \\ \Lc &\mapsto \Lc=\DHom^\punto(\ZZ,\Lc)\endaligned.\]
If $\Lc$ is the constant data $G$ on $X$, then $\Lc=\DHom^\punto(\ZZ,\Lc)$ is also the constant data $G$ on $\widehat X$. For the derived category we obtain  an equivalence 
\[ \aligned D (\text{\rm Qcoh}(X)) &\to  D (\text{\rm Qcoh}(\widehat X))\\ \Lc &\mapsto \Lc=\DHom^\punto(\ZZ,\Lc)\simeq \LL\DHom^\punto(\ZZ,\Lc)
\endaligned
\] 
\end{rem}

\begin{lem}\label{topdualitycomplex} Let $T$ be an  abelian data  on $\widehat X$, $F$ an abelian data on $X$ and $\Lc$ a locally constant data on $X$ (and hence on $\widehat X$). One has: $$\Hom(T,\DHom (F,\Lc) )  =  \Hom(F,\DHom (T,\Lc) ).$$  
More generally, for any   $T\in C(\widehat X)$, $F\in C(X)$ and   $\Lc\in C(\text{\rm Qcoh}(X))= C(\text{\rm Qcoh}(\widehat X))$  one has
\[ \Hom^\punto (T,\DHom^\punto(F,\Lc) ) = \Hom^\punto(F,\DHom^\punto(T,\Lc) ).\]
\end{lem}
\begin{proof} By definition, a morphism $h\colon T\to \DHom (F,\Lc) $ of abelian data is equivalent to give, for each $\widehat p\in\widehat X$, a morphism of groups $h_{\widehat p}\colon T_{\widehat p}\to \DHom (F,\Lc)_{\widehat p}=\Hom(F_p,\Lc_p)$, which are compatible with the restriction morphisms $r_{\widehat p\widehat q}$. But a morphism $h_{\widehat p}\colon T_{\widehat p}\to  \Hom(F_p,\Lc_p)$ is equivalent to a morphism $\widetilde h_p\colon  F_p \to  \Hom(T_{\widehat p},\Lc_p)=\DHom (T,\Lc)_p$, and the compatibility of the $h_{\widehat p}$ with $r_{\widehat p\widehat q}$ is equivalent to the compatibility of the $\widetilde h_p$ with the retriction morphisms $r_{qp}$; thus $\widetilde h_p$ define a morphism of abelian data $F\to \DHom (T,\Lc) $. The extension to complexes is purely formal.
\end{proof}


\begin{thm}\label{topduality} For any $F\in D(X)$, $T\in D(\widehat X)$ and any complex $\Lc$ of locally constant data on $X$ (and hence on $\widehat X$) one has:
 \[\RR\Hom^\punto (T, \LL\DHom^\punto(F,\Lc) ) = \RR\Hom^\punto ( F, \LL \DHom^\punto(T,\Lc) ) .\]

\end{thm}

\begin{proof} Let $P\to F$ and $Q\to T$ be projective resolutions. Then
\[\aligned \RR\Hom^\punto (T, \LL\DHom^\punto(  F,\Lc) ) = \Hom^\punto (Q, \DHom^\punto(  P,\Lc) )& = \Hom^\punto (P, \DHom^\punto(  Q,\Lc) )\\ & = \RR\Hom^\punto ( F, \LL\DHom^\punto( T,\Lc))\endaligned\] where the second equality is due to Lemma \ref{topdualitycomplex}.
\end{proof}

%
%

\begin{cor}\label{topduality2} For any $F\in D(X)$ and any $\Lc\in D(\text{\rm Qcoh}(X))$ one has 
\[ \RR\Gamma(\widehat X,\LL\DHom^\punto(F,\Lc) ) =\RR\Gamma(X,\RR\HHom^\punto(F, \Lc)).
\]
In particular,   
\[ \aligned \RR\Gamma(\widehat X,\Lc)&=  \RR\Gamma(X,\Lc)\\
  \RR\Gamma(X,F^\vee) &=\RR\Gamma(\widehat X,\widehat F ).
\endaligned\] 
\end{cor}

We shall conclude this subsection by generalizing Theorem \ref{topduality} to the relative case (Theorem \ref{topdualityrel}). For the proof, we shall need the following result:

\begin{prop}\label{basechangedual}  Let $f\colon X\to Y$ be a continuous map. For any $F\in C(Y)$ and any $\Lc\in C(\text{\rm Qcoh}(Y))$ one has  a natural isomorphism $${\widehat f}^{-1}\DHom^\punto(G,\Lc) \overset\sim\to \DHom^\punto(f^{-1}G,f^{-1}\Lc).$$
\end{prop}

\begin{proof} For any $\widehat x\in\widehat X$
\[\begin{aligned} \left[ {\widehat f}^{-1}\DHom^\punto(G,\Lc)\right]_{\widehat x} = \DHom^\punto(G,\Lc)_{\widehat f (\widehat x)} =\Hom^\punto (G_{f(x)}, \Lc_{f(x)})& = \Hom^\punto ((f^{-1}G)_x, (f^{-1}\Lc)_x)\\ &=\DHom^\punto((f^{-1}G),f^{-1}\Lc)_{\widehat x} .\end{aligned}\]
\end{proof}



\begin{thm}\label{topdualityrel} Let $f\colon X\to Y$ be a continuous map.\begin{enumerate}
\item For any $G\in D(Y)$, $\Lc\in D(\text{\rm Qcoh}(Y))$, one has a natural isomorphism
\[ {\widehat f}^{-1} \LL\DHom^\punto(G,\Lc) \overset\sim\longrightarrow    \LL\DHom^\punto(f^{-1}G,f^{-1}\Lc).\]
\item For any $F\in D(X)$,  $\Lc\in D(\text{\rm Qcoh}(Y))$, one has a natural isomorphism
\[ \LL\DHom^\punto(\LL f_! F, \Lc)\overset\sim\longrightarrow \RR {\widehat f}_*\, \LL\DHom^\punto(F,f^{-1}\Lc).\] In particular, \[ \widehat{\LL f_! F}\simeq \RR {\widehat f}_*\widehat F.\]
\end{enumerate}
\end{thm}

\begin{proof} (1) Let $P\to G$ and $Q\to f^{-1}P$ be  projective resolutions. By Proposition \ref{basechangedual} one has an isomorphism $ {\widehat f}^{-1}  \DHom^\punto(P, \Lc)\simeq  \DHom^\punto(f^{-1} P, f^{-1}\Lc) $; moreover, the natural morphism  $\DHom^\punto(f^{-1} P, f^{-1}\Lc) \to  \DHom^\punto(Q, f^{-1}\Lc)$ is an isomorphism (in the derived category) by Proposition \ref{puntuallyprojective}. Thus, we have an isomorphism  \[
  {\widehat f}^{-1} \LL\DHom^\punto(G, \Lc) \to \LL\DHom^\punto(f^{-1}G, f^{-1}\Lc).\] (2)  Taking $\RR\Hom^\punto(T,\quad)$ for any $T\in D(\widehat X)$ in the isomorphism of item (1), one obtains (by adjunctions and Theorem \ref{topduality}) an isomorphism
 \[ \RR\Hom^\punto (G, \LL\DHom^\punto(\LL {\widehat f}_! T,\Lc)) \overset\sim  \longrightarrow \RR\Hom^\punto ( G  ,\RR f_*\,  \LL\DHom^\punto( T,{\widehat f}^{-1}\Lc))\]
 i.e., an isomorphism $ \LL\DHom^\punto(\LL {\widehat f}_! T, \Lc)  \to \RR f_*(  \LL\DHom^\punto(  T, {\widehat f}^{-1}\Lc) $. Replacing $f\colon X\to Y$ by $\widehat f\colon \wh X\to \wh Y$, we conclude. 
%
\end{proof}


\subsection{Duality and Coduality}

\medskip

\subsubsection{Grothendieck-Verdier duality for cohomology}

Finite spaces admit a Verdier duality theorem: the functor $\RR f_*$ admits a right adjoint $f^!$;  in particular, one has relative and absolute dualizing complexes. After giving some basic results on the structure of the dualizing complex, we shall study  the role that c-proper maps have in duality theory; the main result is Theorem \ref{projformulacharacterization} that says that a map $f$ is c-proper if and only if the local isomorphism of duality holds and if and only if the   functor $f^!$ is local on $Y$. Thus, c-proper maps play, again,  the same role that proper maps do for locally compact and separated spaces.

\begin{thm} [\cite{Navarro}] Let $f\colon X\to Y$ be a continuous map. The functor $$\RR f_*\colon D(X)\to D(Y)$$ has a right adjoint (that shall be denoted by $f^!$).

\end{thm}

\begin{proof} We shall give a sketch of a (slightly modified) proof and refer to \cite{Navarro} for the details, because the construction of the functor $f^!$ will be used in the sequel. 

 Let $G$ be an abelian data on $Y$. The functor \[ \aligned \{\text{\rm Abelian data on }X\} &\to \{ \text{\rm Abelian groups}\}\\ F&\mapsto \Hom(f_*C^nF,G)
\endaligned\] is right exact and takes filtered direct limits into filtered inverse limits. Hence it is representable. Let $f^{-n}G$ be the representant. A morfism $G\to G'$ induces a morphism $f^{-n}G\to f^{-n}G'$. The natural morphism $f_*C^{n-1}F\to f_*C^nF$ induces a morphism $f^{-n}G\to f^{-n+1}G$. Thus, if $G$ is a complex of abelian data on $Y$, then we have a bicomplex $f^{-p}G^q$, whose associated simple complex is denoted by $f^\nabla G$. One has an isomorphism
\[ \Hom^\punto (f_*C^\punto F,G)=\Hom^\punto(F,f^\nabla G)\] and then it suffices to define $f^!G:=f^\nabla I(G)$, with $G\to I(G)$ an injective resolution.
\end{proof}

\begin{defn} We shall denote $$D_{X/Y}:=f^!\ZZ$$ and call it {\it relative dualizing complex} of $X$ over $Y$. If $Y$ is a point, it will be denoted by $D_X$ and named by {\it  dualizing complex of $X$.}
\end{defn}

\begin{rem} By definition, for any $F\in D(X)$ one has:

\[\aligned \RR\Hom^\punto(\RR f_*F,\ZZ)&= \RR\Hom^\punto (F, D_{X/Y})\\ \RR\Gamma(X,F)^\vee &=\RR\Hom^\punto(F,D_X).\endaligned\]
In particular, $\RR\Gamma(U,D_X)=\RR\Gamma(X,\ZZ_U)^\vee$. Thus the stalkwise description of $D_X$ is: 
\[ (D_X)_p=\RR\Gamma(X,\ZZ_{U_p})^\vee\] and, for any $p\leq q$, the restriction morphism
$(D_X)_p\to (D_X)_q$ is obtained by applying $\RR\Gamma(X,\quad)^\vee$ to the natural morphism $\ZZ_{U_q}\hookrightarrow \ZZ_{U_p}$.
\end{rem}

\begin{thm} [Explicit description of the dualizing complex]\label{ExplicitDualizing} Let $n=\dim X$. The dualizing complex $D_X$ is isomorphic to a complex 
\[ 0\to D_X^{-n}\to\cdots\to D_X^0\to 0\] where $$D_X^{-p}=\underset{x_0<\cdots <x_p}\bigoplus \ZZ_{C_{x_p}}.$$ In particular, $D_X$ is of finite type and has amplitude $[-n,0]$.
\end{thm}

\begin{proof}  Let $\Omega$ be the standard resolution of $\ZZ$ by injective abelian  groups:  $\Omega^0=\QQ\to \Omega^1=\QQ/\ZZ$. Let $F$ be an abelian data. Then $\RR\Gamma (X, F)$ is the complex
\[ 0\to \Gamma(X,C^0F)\to \Gamma(X,C^1F)\to \cdots \to \Gamma(X,C^nF)\to 0\] and  $\Gamma(X,C^pF)=\underset {x_0<\cdots <x_p}\bigoplus F_{x_p}$. 
For each point $x\in X$, let $i\colon{x}\hookrightarrow X$ be the inclusion. One has that $i_*G=G_{C_x}$ for any abelian group $G$, hence   $$(F_x)^\vee=\Hom^\punto (F_x,\Omega) = \Hom^\punto (F,i_*\Omega)=\Hom^\punto (F,\Omega_{C_x})$$ and then
\[ \Gamma(X,C^pF)^\vee=\Hom^\punto(F, \underset{x_0<\cdots <x_p} \bigoplus \Omega_{C_{x_p}}).\] One concludes because $\Omega_{C_{x_p}}$ is an injective resolution of $\ZZ_{C_{x_p}}$.
\end{proof}

\begin{rem} One can easily compute the differential of the complex $D_X$, which is induced by the differential of $\Gamma (X,C^\punto F)$.
\end{rem}

\begin{prop} If $X$ has a minimum $p$, then $D_X\simeq \ZZ_{\{ p\}}$.
\end{prop}

\begin{proof} Let $i\colon \{ p\}\hookrightarrow X$ be the inclusion. For any abelian data $F$ on $X$ one has
\[\Gamma(X,F)=\Gamma(U_p,F)=F_p=i^{-1}F\] Hence
\[\RR\Hom^\punto(F,D_X)=\RR\Hom(\Gamma(X,F),\ZZ)=\RR\Hom(i^{-1}F,\ZZ)= \RR\Hom( F,\RR i_*\ZZ)= \RR\Hom( F,\ZZ_{\{ p\}}).\]
\end{proof}

%

\begin{prop}\label{localdualizing} Let $f\colon X\to Y$ be a continuous map, $F\in D(X)$, $K\in D(Y)$. 
\begin{enumerate} \item One has a natural morphism 
\[ \RR f_*\RR\HHom^\punto(F,f^!K)\to \RR\HHom^\punto(\RR f_*F,K)\]
which is called {\it local homomorphism of duality}. Taking global sections (i.e. applying $\RR\Gamma(Y,\quad)$) one obtains the duality isomorphism.

\item For any open subset $V\subseteq Y$, one has a natural morphism
\[ (f^!K)_{\vert f^{-1}(V)}\to {f_V}^!(K_{\vert V}).\]
\end{enumerate}
\end{prop}

\begin{proof}
  For any $F\in D(X)$, $\Lc\in D(Y)$ one has the projection formula  morphism:
\[\ \RR f_*F\overset\LL\otimes \Lc\to\RR f_*(F\overset\LL\otimes f^{-1}\Lc).\]
Taking $\RR\Hom^\punto(\quad,K)$, one obtains, by adjunctions, a morphism
\[ \RR\Hom^\punto(\Lc, \RR f_*\RR\HHom^\punto(F,f^!K))\to \RR\Hom^\punto(\Lc,\RR\HHom^\punto(\RR f_*F,K))\] hence a morphism
\[ \RR f_*\RR\HHom^\punto(F,f^!K)\to \RR\HHom^\punto(\RR f_*F,K).\]
Taking sections on $V$ one obtains a morphism (let us denote $U=f^{-1}(V)$)
\[  \RR\Hom^\punto(F_{\vert U},(f^!K)_{\vert U})\to \RR\Hom^\punto(\RR {f_V}_*(F_{\vert U}),K_{\vert V})\overset{\text{\rm Duality}}{===}\RR\Hom^\punto(F_{\vert U},{f_V}^!(K_{\vert V})\] hence a morphism 
\[ (f^!K)_{\vert U}\to {f_V}^!(K_{\vert V}).\]
\end{proof}


\begin{prop}\label{4.3.9} Let $f\colon X\to Y$ be a continuous map. For any  $\Lc,K\in D(Y)$ one has a natural morphism 
 \[  \RR\HHom^\punto(f^{-1}\Lc, f^!K) \to  f^!\RR\HHom^\punto(\Lc,K) \] which is an isomorphism if $\Lc$ is locally constant.
\end{prop}
\begin{proof}
Taking $\RR\Hom^\punto(\quad,K)$ in the projection formula  one obtains a morphism
 \[ \RR\Hom^\punto(F,\RR\HHom^\punto(f^{-1}\Lc, f^!K))\to \RR\Hom^\punto(F,f^!\RR\HHom^\punto(\Lc,K))\] hence a morphism
  \[  \RR\HHom^\punto(f^{-1}\Lc, f^!K) \to  f^!\RR\HHom^\punto(\Lc,K) \] which is an isomorphism if $\Lc$ is locally constant, because projection formula is so in that case.
\end{proof}
 
%

  
\begin{thm}\label{projformulacharacterization} Let $f\colon X\to Y$ be a continuous map. The following conditions are equivalent
\begin{enumerate} \item $f$ is c-proper.

\item $f^!$ is local on $Y$: for any open subset $V\subseteq Y$ and any $K\in D(Y)$ the morphism (Proposition \ref{localdualizing}) \[ (f^!K)_{\vert f^{-1}(V)}\to f_V^!(K_{\vert V})\] is an isomorphism.

\item Local isomorphism of duality holds: for any $F\in D(X), K\in D(Y)$, the local morphism of duality (Proposition \ref{localdualizing})
 \[\RR f_*\RR\HHom^\punto(F,f^!K) \to \RR\HHom^\punto (\RR f_*F,K)\] is an isomorphism.
 
 \item Commutation of $f^!$ with homomorphisms: for any $\Lc,K\in D(Y)$ the morphism (Proposition \ref{4.3.9})
 \[  \RR\HHom^\punto(f^{-1}\Lc, f^!K) \to  f^!\RR\HHom^\punto(\Lc,K) \] is an isomorphism.

\end{enumerate}
\end{thm}

\begin{proof} By Theorem \ref{hsatprojform}, (1) is equivalent to saying that projection formula holds. Then 
 (3) and (4) are equivalent to (1) by adjunction. It is also clear that (3) implies (2). To conclude, let us see that (2) implies (3). In order to prove that
 \[\RR f_*\RR\HHom^\punto(F,f^!K) \to \RR\HHom^\punto (\RR f_*F,K)\] is an isomorphism, let us take sections on an open subset $V$ of $Y$. We obtain the morphism
 \[ \aligned \RR\Hom^\punto(F_{\vert f^{-1}(V)},&(f^!K)_{\vert f^{-1}(V)})\simeq \RR\Gamma(V, \RR f_*\RR\HHom^\punto(F,f^!K)) \to \RR\Gamma(V, \RR\HHom^\punto (\RR f_*F,K))= \\ &= \RR\Hom^\punto ((\RR f_*F)_{\vert V},K_{\vert V})\simeq \RR\Hom^\punto (\RR {f_V}_*(F_{\vert f^{-1}(V)}),K_{\vert V})\overset{\text{\rm Duality}}{===}\\ &= \RR\Hom^\punto ( F_{\vert f^{-1}(V)}),f_V^!(K_{\vert V})) \endaligned \] which is an isomorphism because $(f^!K)_{\vert f^{-1}(V)}\to f_V^!(K_{\vert V})$ is an isomorphism by hypothesis.

\end{proof}

%
%

\subsubsection{Co-Duality: Grothendieck-Verdier duality for homology}
\medskip

 In this subsection we shall provide a Grothendieck-Verdier duality theorem for homology. In particular, we shall obtain relative and absolute codualizing complexes. We shall give the basic properties of the codualizing complex and study the significance of h-open maps in co-duality. Finally, we shall make a brief study of the relation between the dualizing and codualizing complexes.
 \medskip
 
  Our first   aim  is to prove the following:
 
\begin{thm}\label{coduality} Let $f\colon X\to Y$ be a continuous map.
The functor $\LL f_!$ has a left adjoint: $f^\#$.
\end{thm}

For the proof, we shall need some previous results.

%

\begin{lem}\label{4.3.12} Let $\{ F_i\}$ be an  inverse system of abelian data on $X$. If $X$ has the discrete topology, then the natural morphism
\[ L(X,\plim{}F_i)\to \plim{}L(X,F_i)\] is an isomorphism.
\end{lem}

\begin{proof} Since $X$ has the discrete topology,  
\[ L(X,F)=\underset{x\in X}\oplus F_x\] for any abelian data $F$ on $X$. One concludes because $X$ is finite.
\end{proof}

\begin{prop}\label{prop2} Let $f\colon X\to Y$ be a continuous map and $\{ F_i\}$ an  inverse system of abelian data on $X$. If $X$ has the discrete topology, the natural morphism
\[ f_!\plim{} F_i\to \plim{} f_!F_i\] is an isomorphism.
\end{prop}

\begin{proof} For any $y\in Y$,
\[ (f_!\plim{} F_i)_y = L(f^{-1}(C_y), \plim{} F_i) \overset{\ref{4.3.12}}\simeq \plim{} L(f^{-1}(C_y),  F_i) = \plim{} (f_!F_i)_y= [\plim{} f_!F_i]_y.\]
\end{proof}

\begin{prop}\label{plimC_n} Let $f\colon X\to Y$ be a continuous map and $\{ F_i\}$ an  inverse system of abelian data on $X$. The natural morphism
\[ f_!C_n(\plim{} F_i) \to \plim{} f_!C_n  F_i\] is an isomorphism.
\end{prop}
\begin{proof} By Remark \ref{C^n&C_n},  
for any abelian data $F$ on $X$, one has $$C_nF={\pi_n}_!C_0(\pi_0^{-1}F)={\pi_n}_!\id_!\id^{-1}\pi_0^{-1}F,$$ so
\[ f_!C_nF = (f\circ\pi_n\circ\id)_!(\pi_0\circ\id)^{-1}F.\]
Conclusion follows fom Proposition  \ref{prop2}.
\end{proof}

\begin{prop} Let $f\colon X\to Y$ be a continous map, $G$ an abelian data on $Y$ and $p\geq 0$ an integer. The functor
\[ \aligned \left\{\text{\rm Abelian data on }X\right\}&\to \left\{\text{\rm Abelian Groups}\right\}\\ F &\mapsto \Hom(G, f_!C_pF)\endaligned
\] is representable.
\end{prop}

\begin{proof} This functor is left exact, commutes with finite direct products and with inverse limits (by Proposition \ref{plimC_n}). Hence it is representable.
\end{proof}

By definition, there exists an abelian data $f_{-p}G$ on $X$ and a morphism $\epsilon\colon G\to f_!C_p(f_{-p}G)$ such that, for any abelian data $F$ on $X$, the map
\[ \aligned \Hom(f_{-p}G,F)&\to \Hom (G,f_!C_pF)\\ h&\mapsto f_!C_p(h)\circ\epsilon
\endaligned\] is an isomorphism. A morphism $G\to G'$ induces a morphism $f_{-p}G\to f_{-p}G'$, thus $f_{-p}$ defines a functor from abelian data on $Y$ to abelian data on $X$. Moreover, the natural morphism $C_pF\to C_{p-1}F$ induces a morphism $f_{-p+1}G\to f_{-p}G$.

\begin{defn} Let $G$ be a complex of abelian data on $Y$. We shall denote $f^\square G$ the simple complex associated to the bicomplex $f_{-p}G_q$.
\end{defn}

\begin{prop}\label{dualitybydegree} For any $G\in C(Y)$ and any $F\in C(X)$ one has an isomorphism
\[ \Hom_\punto (f^\square G,F)\overset\sim\to \Hom_\punto(G,f_!C_\punto F)\]
\end{prop}
 
\begin{proof} \[\aligned \Hom_n(G,f_!C_\punto F)  =&\underset p\prod \Hom(G_{p+n},f_!C_p F )   =  \underset p\prod \Hom(G_{p+n},\underset i\prod f_!C_i F_{p-i} ) \\ = &\underset {p,i}\prod  \Hom(G_{p+n}, f_!C_i F_{p-i})   \overset\sim \leftarrow   \underset {p,i}\prod \Hom(f_{-i}G_{p+n},  F_{p-i} )\\ \overset{q=p-i}{=\hspace{ -1.5 pt} =\hspace{ -1.5 pt} =\hspace{ -1.5 pt} =}& \underset {q,i}\prod \Hom(f_{-i}G_{q+i+n},  F_{q} )\  =  \underset q\prod  \Hom(\underset i\bigoplus f_{-i}G_{q+i+n},  F_{q} )\\  = &\underset q\prod  \Hom((f^\square G)_{q+n},  F_{q} )   =  
\Hom_n (f^\square G,F) \endaligned\]
\end{proof}

\begin{cor} If $G$ is projective (resp. homotopic to zero), then $f^\square G$ is also projective (resp., homotopic to zero). 
\end{cor}

\begin{defn} For any $G\in D(Y)$, we shall denote $f^\#G:= f^\square P$, where $P\to G$ is a projective resolution. This defines a functor
\[ f^\#\colon D(Y)\to D(X)\] such that
\[ \RR\Hom^\punto (f^\# G, F)=\RR\Hom^\punto (G,\LL f_! F)\] for any $G\in D(Y)$, $F\in D(X)$. In particular, $f^\#$ is a left adjoint of $\LL f_!$ and Theorem \ref{coduality} is proved.
\end{defn}



\begin{proof} For any $F\in D(X)$, \[\aligned  \RR\Hom^\punto (f^\#  (K\overset\LL\otimes\Lc), F)& =  \RR\Hom^\punto ( K\overset\LL\otimes\Lc,\LL f_! F) =  \RR\Hom^\punto ( K, \Lc^\vee\overset\LL\otimes \LL f_! F)=\\ & =  \RR\Hom^\punto ( K,  \LL f_!(f^{-1}\Lc^\vee\overset\LL\otimes F)) = 
 \RR\Hom^\punto ( f^\# K, f^{-1}\Lc^\vee\overset\LL\otimes F) = \\ &= \RR\Hom^\punto ( f^\# K\overset\LL\otimes f^{-1}\Lc, F).\endaligned \]
\end{proof}

\begin{defn} Let $f\colon X\to Y$ be a continuous map. We shall denote  $$D^{X/Y}=f^\#\ZZ$$ and name it {\it relative codualizing complex of $X$ over $Y$}. In the case that $Y$ is a point, it will be called {\it codualizing complex of $X$} and denoted by $D^X$. By definition one has
\[ \aligned \RR\Hom(D^{X/Y},F)&=\RR\Gamma(Y,\LL f_!F)\\ \RR\Hom(D^X,F)&= \LL(X,F).
\endaligned\]
\end{defn}

\begin{thm}[Explicit description of the codualizing complex]\label{ExplicitCodualizing}  Let $n=\dim X$. The codualizing complex $D^X$ is isomorphic to a complex
\[ 0\to (D^X)^0\to\cdots\to (D^X)^n\to 0\] with   \[ (D^X)^i=\underset{x_0<\cdots <x_i}\bigoplus \ZZ_{U_{x_0}}. \] 
In particular, $D^X$ is of finite type.
\end{thm}

\begin{proof}  Let $F$ be an abelian data. Then $\LL (X, F)$ is the complex
\[ 0\to L(X,C_nF)\to\cdots \to L(X,C_1F)\to L(X,C_0F)\to 0\] and  $L(X,C_pF)=\underset {x_0<\cdots <x_p}\bigoplus F_{x_0}$.
For each point $x\in X$, one has
 $$F_x=\Gamma(U_x,F)= \Hom(\ZZ_{U_x}, F)$$ and then
\[ L(X,C_pF)=\Hom( \underset{x_0<\cdots <x_p} \bigoplus \ZZ_{U_{x_0}}, F).\] One concludes because $\ZZ_{U_x}$ is projective.
\end{proof}

\begin{rem} One can easily compute the differential of the complex $D^X$, which is induced by the differential of $L(X,C_\punto F)$.
\end{rem}

\begin{cor}    For any closed subespace $Y$ of $X$ one has
\[ \LL (Y, D^X) = \RR\Gamma_Y(X,\ZZ).\] In particular, $\LL (X, D^X) = \RR\Gamma (X,\ZZ).$
\end{cor}

\begin{proof}  $\LL (Y, D^X)^\vee = \RR\Hom (D^X,\ZZ_Y)=\LL (X, \ZZ_Y)$. Taking dual, 
\[ \LL (Y, D^X)  =   \LL (X, \ZZ_Y)^\vee =\RR\Gamma_Y(X,\ZZ).\]
\end{proof}

\begin{cor} For each $p\in X$ one has:
\[ (D^X)_p=\RR\Gamma_{C_p}(X,\ZZ)\] Thus, $H^i(D^X)_p=H^i_{C_p}(X,\ZZ)$.
\end{cor}

\begin{proof} $[ (D^X)_p]^\vee =\LL (C_p, D^X)^\vee = \RR \Hom (D^X,\ZZ_{C_p}) = \LL (X,\ZZ_{C_p})$. Taking dual,
\[  (D^X)_p= \LL (X,\ZZ_{C_p})^\vee =\RR\Gamma_{C_p}(X,\ZZ).\]
\end{proof}

\begin{prop} If $X$ is irreducible (i.e., it has a maximum $g$), then $D^X\simeq \ZZ_{\{ g\}}$.
\end{prop}

\begin{proof} Let $i\colon \{ g\}\hookrightarrow X$ be the inclusion. For any abelian data $F$ on $X$ one has
\[L(X,F)=L(C_g,F)=F_g=i^{-1}F.\] Hence
\[\RR\Hom^\punto(D^X,F)=\RR\Hom^\punto(\ZZ,\LL(X,F)) =  \RR\Hom^\punto (\ZZ, i^{-1}F) = \RR\Hom^\punto(i_!\ZZ, F)\] and one concludes because $i_!\ZZ= \ZZ_{\{ g\}}$.
\end{proof}

\begin{prop}\label{localcodualizing} Let $f\colon X\to Y$ be a continuous map. For any  $K\in D(Y)$ and 
 any closed subset $C\subseteq Y$, one has a natural morphism
\[ {f_C}^\#(K_{\vert C}) \to (f^\#K)_{\vert f^{-1}(C)}.\]
\end{prop}

\begin{proof} Let us denote $i\colon C\hookrightarrow Y$ and $j\colon f^{-1}(C)\hookrightarrow X$ the inclusions and let $F\in D(f^{-1}(C))$. Taking $j_*$ in the natural map $F\to f_C^{-1}\LL {f_C}_! F$, one obtains a morphism  
\[ j_*F\to j_*  f_C^{-1}\LL {f_C}_! F\overset\sim\to f^{-1}i_* \LL {f_C}_! F\] hence a morphism
\[ \LL f_!j_*F\to   i_* \LL {f_C}_! F.\] Taking $\RR\Hom^\punto(K,\quad)$ one obtains a morphism
\[ \RR\Hom^\punto((f^\#K)_{\vert f^{-1}(C)}, F)\to  \RR\Hom^\punto({f_C}^\#(K_{\vert C}), F)\] and one concludes.
\end{proof}

\begin{rem} The adjunction between $f^{-1}$ and $\RR f_*$  (Proposition \ref{AdjunctionsOfInvIm}) can be sheafified to an isomorphism \[ \RR\HHom^\punto (G,\RR f_*F)\simeq \RR f_* \RR\HHom^\punto (f^{-1}G,F).\] 

However, the adjunction between $\LL f_!$  and $f^{-1}$  can be sheafified to a morphism 
\[  \RR\HHom^\punto ( \LL f_!F,G)\to \RR f_* \RR\HHom^\punto (F,f^{-1}G) \] which is not an isomorphism in general (indeed, we shall see in Theorem \ref{coprojformulacharacterization}  that this isomorphism holds if and only if $f$ is h-open). In fact, for any $K\in D(Y)$, taking $\RR\Hom^\punto(\quad,G)$ in the coprojection formula
\[  \LL f_!(F\overset\LL \otimes  f^{-1} K) \to (\LL f_!F)\overset\LL\otimes K\] gives a morphism
\[ \RR\Hom^\punto(K, \RR\HHom^\punto ( \LL f_!F,G))\to\RR\Hom^\punto(K, \RR f_* \RR\HHom^\punto (F,f^{-1}G))\]
i.e, a morphism 
\[  \RR\HHom^\punto ( \LL f_!F,G)\to \RR f_* \RR\HHom^\punto (F,f^{-1}G). \]
\end{rem}

\begin{thm}\label{coprojformulacharacterization} Let $f\colon X\to Y$ be a continuous map. The following conditions are equivalent
\begin{enumerate} \item $f$ is h-open.

\item Local adjunction isomorphism between $\LL f_!$ and $f^{-1}$ holds: for any $F\in D(X), G\in D(Y)$, the natural morphism \[ \RR\HHom^\punto(\LL f_!F , G ) \to \RR f_*  \RR\HHom^\punto (F, f^{-1} G)\] is an isomorphism.
 
\item  $f^{-1}$ commutes with homomorphisms: for any $\Lc,G\in D(Y)$ the morphism
\[ f^{-1}\RR\HHom^\punto(\Lc,G)\to \RR\HHom^\punto(f^{-1}\Lc, f^{-1}G)    \] is an isomorphism.
\item $f^\#$ is local on $Y$: for any closed subset $C\subseteq Y$ and any $K\in D(Y)$ the natural morphism \[  (f_C)^\#(K_{\vert C}) \to (f^\#K)_{\vert f^{-1}(C)}\] is an isomorphism.
\end{enumerate}
\end{thm}

\begin{proof} By Theorem \ref{hsatprojform}, (1) is equivalent to saying that coprojection formula holds.  Then, the equivalence of (1), (2) and (3) follows from adjunction: one has a commutative diagram
\[\xymatrix {\RR\Hom^\punto( \Lc, \RR\HHom^\punto(\LL f_! F, G) )\ar[r]\ar[d]^\wr & \RR\Hom^\punto(    \Lc, \RR f_*\RR\HHom^\punto(F, f^{-1} G)) \ar[d]^\wr \\ 
\RR\Hom^\punto( (\LL f_! F)\overset\LL\otimes \Lc, G )\ar[r] \ar[d]^\wr & \RR\Hom^\punto(\LL f_!(F\overset\LL \otimes f^{-1}\Lc),G) \ar[d]^\wr\\ 
\RR\Hom^\punto( F,f^{-1}\RR\HHom^\punto( \Lc, G) )\ar[r] & \RR\Hom^\punto( F,\RR\HHom^\punto(  f^{-1}\Lc,f^{-1} G)) .   }
\]

(1) $\Rightarrow$ (4).
 Let us denote $i\colon C\hookrightarrow Y$, $j\colon f^{-1}(C)\hookrightarrow X$. By Theorem \ref{hsatprojform}, $\LL f_!$  commutes with supporting on closed subsets. Hence, $\LL f_!(F_{f^{1}(C)})\simeq (\LL f_!F)_C$ for any $F\in D(X)$. Taking $\RR\Hom(K,\quad)$, one obtains
\[\RR\Hom^\punto(K, \LL f_!(F_{f^{1}(C)}))\simeq \RR\Hom^\punto(K,(\LL f_!F)_C)\]
Now,
\[ \RR\Hom^\punto(K, \LL f_!(F_{f^{1}(C)}))= \RR\Hom^\punto(f^\#K,  F_C ) = \RR\Hom^\punto(j^{-1}f^\#K,  j^{-1}F  )\]
and 
\[ \aligned \RR\Hom^\punto(K, (\LL f_!F)_C)= \RR\Hom^\punto(i^{-1}K, i^{-1}\LL f_!F )& = \RR\Hom^\punto(i^{-1}K, \LL {f_C}_! j^{-1}F ) \\ & = \RR\Hom^\punto((f_C)^\#i^{-1}K,   j^{-1}F )\endaligned\] and we conclude that $j^{-1}f^\#K\simeq (f_C)^\# i^{-1}K$.

(4) $\Rightarrow$ (1). The above isomorphisms prove that the isomorphism $j^{-1}f^\#K\simeq (f_C)^\# i^{-1}K$  implies an isomorphism $\LL f_!(F_{f^{1}(C)})\simeq (\LL f_!F)_C$. One concludes by Theorem \ref{hsatprojform}.
\end{proof}

\begin{thm}\label{dualizing-codualizing} One has a natural morphism 
\[ \RR\Hom(D^X,F)\to \RR\Hom(F^\vee, D_X)\] which is an isomorphism if $F$ is of finite type. In other words, there exists a morphism
\[ \xi\colon (D^X)^\vee\to D_X\] such that, the induced morphism
\[ \aligned \RR\Hom(D^X,F)&\to \RR\Hom(F^\vee, D_X)\\ h&\mapsto \xi\circ h^\vee
\endaligned\] is an isomorphism if $F$ is of finite type.
\end{thm}

\begin{proof} For any $F$ one has
\[ \RR\Hom (D^X,F)=\LL (X,F)\to \LL (X,F)^{\vee\vee} = \RR\Gamma(X,F^\vee)^\vee=\RR \Hom(F^\vee,D_X)\] and $\LL (X,F)\to \LL (X,F)^{\vee\vee}$ is an isomorphism if $F$ is of finite type.
\end{proof}

The following definition is inspired by its algebro-geometric analogous.

\begin{defn} We say that $X$ is {\it homologically  Gorenstein} (resp. {\it cohomologically  Go\-ren\-stein}) if 
\[ D^X\simeq \T^X[-d]\qquad (\text{\rm resp. } D_X\simeq \T_X[d])\] for some integer $d\geq 0$ and some invertible abelian data $\T^X$ (resp. $\T_X$).
\end{defn}

\begin{ejem} Let $n\geq 1$ and let $X=\Sc^n (S^0)$, which is the minimal finite model of the $n$-dimensional sphere $S^n$ (see \cite{Barmak}). Then
\[ D^{\Sc^n(S^0)}\simeq \ZZ[-n].\]

\begin{proof} For brevity, let us denote  $\Sc^n=\Sc^n(S^0)$. For any $p\in \Sc^n$,  $\Sc^n-C_p$ is contractible. Hence  
\[ H^i[(D^X)_p]=H^i_{C_p}(\Sc^n ,\ZZ)=\left\{\aligned 0\qquad &\text{, if } i\neq n \\ H^n(\Sc^n,\ZZ)=\ZZ &\text{, if } i= n,\endaligned\right.\] and for any $p\leq q$ the morphism 
\[ \ZZ= H^i[(D^X)_p]\to \ZZ= H^i[(D^X)_q]\] is the identity. It follows that $D^X\simeq \T^X[-n]$, with $\T^X$ an invertible abelian data. From the isomorphism $\RR\Hom^\punto (D^X,\ZZ)=\LL(X,\ZZ)$, it follows that $\Hom(\T^X,\ZZ)=H_n(\Sc^n,\ZZ)=\ZZ$, and then $\T^X\simeq\ZZ$ and  $D^X=\ZZ[-n]$.
\end{proof}
\end{ejem}

Notice that if $X$ is homologically Gorenstein then $D^X$ is perfect. The following Proposition says that the converse is also true and that being homologically Gorenstein is stronger than being cohomologically Gorenstein.

\begin{prop}\label{perfect} If $X$ is connected and  $D^X$ is perfect, then $X$ is homologically Gorenstein and $(D^X)^\vee\to D_X$ is an isomorphism. In particular, $X$ is also cohomologically Gorenstein and $\HHom(\T^X,\ZZ) \simeq \T_X$.
\end{prop}

\begin{proof} Let us denote $\Lc=D^X$. Since $\Lc$ is perfect, one has: $\LL (X,F)=\RR\Hom(\Lc,F) = \RR\Gamma(X, F\overset\LL\otimes \Lc^\vee)$. Taking dual, one obtains $\RR\Gamma(X,F^\vee)=\RR\Hom(F\overset\LL\otimes \Lc^\vee,D_X)$. This can be translated into an isomorphism
\[ \RR\Hom(F,\ZZ)= \RR\Hom(F,D_X\overset\LL\otimes \Lc),\] and then $D_X\overset\LL\otimes \Lc\simeq \ZZ$. Thus, for any point $p\in X$, one has that $(D_X)_p\overset\LL\otimes \Lc_p\simeq \ZZ$.  By Lemma \ref{lemagrupos}, one has that $\Lc_p\simeq \ZZ[-d]$ and then, since $X$ is connected, $\Lc \simeq \T^X[-d]$ for some integer $d$ and some invertible abelian data $\T^X$. Consequently, $D_X\simeq\Lc^\vee\simeq \T_X[d]$, with $\T_X=\HHom(\T^X,\ZZ)$.
\end{proof}

\begin{lem}\label{lemagrupos} Let $K,L$ be two bounded complexes of finitely generated abelian groups. If $K\overset\LL\otimes L=\ZZ$, then $K\simeq \ZZ[n]$ for some integer $n$ (and then $L\simeq \ZZ[-n]$).
\end{lem}

\begin{proof} It is an easy exercise taking into account that any bounded complex of abelian groups is isomorphic (in the derived category) to its cohomology: $K\simeq \underset i\oplus  H^i(K)[-i]$.
\end{proof}


\begin{thebibliography}{1}


\bibitem{Alexandroff} {\sc P.~Alexandroff}, {\em Diskrete räume},Mat. Sb., 2 (1937), pp. 501-518.

 

\bibitem{Barmak} {\sc J.A.~Barmak}, {\em Algebraic Topology of Finite Topological Spaces and Applications}, Lecture Notes in Mathematics, 2032. Springer, Heidelberg (2011).

\bibitem{BjörnerWachsWelker} {\sc A.~Björner, M.L.~Wachs and V.~Welker}, {\em Poset fiber theorems}, Trans. Amer. Math. Soc.  357  (2005),  no. 5, 1877–1899. 


\bibitem{Bredon} {\sc G.E.~Bredon}, {\em Sheaf theory}, Second edition. Graduate Texts in Mathematics, 170. Springer-Verlag, New York, 1997. xii+502 pp.


\bibitem{Curry} {\sc J.M.~Curry}, {\em Dualities between cellular sheaves and cosheaves},
J. Pure Appl. Algebra  222 (2018), no. 4, 966-993.


\bibitem{Deheuvels} {\sc R.~Deheuvels}, {\em Homologie des ensembles ordonn\'es et des espaces topologiques},
Bull. Soc. Math. France 90, 261-322. (1962).

%
%
%

\bibitem{Iversen} {\sc B.~Iversen}, {\em  Cohomology of sheaves}, Lecture Notes Series, 55. Aarhus Universitet, Matematisk Institut, Aarhus, 1984. vi+237 pp.



\bibitem{McCord} {\sc M.C.~McCord}, {\em Singular homology groups and homotopy groups of finite topological spaces}, Duke Math. J. 33 (1966), 465-474.
 

\bibitem{Navarro} {\sc J.A.~Navarro Gonz{\'a}lez}, {\em Duality and finite spaces}, Order  6  (1990),  no. 4, 401–408. 
 
%


\bibitem{Quillen} {\sc D.~Quillen}, {\em Homotopy Properties of the Poset of Nontrivial $p$-Subgroups of a Group}, Adv.  Math. 28 (1978), 101-128.

\bibitem{Sancho} {\sc F.~Sancho de Salas}, {\em Homotopy of finite ringed spaces}, J. Homotopy Relat. Struct. (2017), https://doi.org/10.1007/s40062-017-0190-2.

%
%

\bibitem{Stong} {\sc R.E.~Stong}, {\em Finite topological spaces}, Trans. Amer. Math. Soc. 123 (1966), 325-340.



\end{thebibliography}
\end{document}